\documentclass[12pt]{amsart}
\usepackage{psfrag,epsfig,amssymb}
\numberwithin{figure}{section}
\numberwithin{table}{section}

\newtheorem{theorem}{Theorem}[section]

\newtheorem{lemma}[theorem]{Lemma}
\newtheorem{prop}[theorem]{Proposition}

\theoremstyle{definition}

\newtheorem{example}[theorem]{Example}
\newtheorem{cor}[theorem]{Corollary}
\newtheorem{question}[theorem]{Question}

\theoremstyle{remark}
\newtheorem{remark}[theorem]{Remark}
\newtheorem*{notation}{Notation}
\numberwithin{equation}{section}

\def \r-equiv{\stackrel{r}{ \simeq} }
\def \Z{{\mathbb Z}}

\def \H{{\mathbb H}}

\def \Q{{\mathbb Q}}
\def \rank{\mathop{\rm rk}}
\def \h{{\mathfrak h}}
\def \g{{\mathfrak g}}

\def \Isom{\rm{Isom}}
\def \rr{{\it r}}

\begin{document}


\author{Anna Felikson}
\address{Independent University of Moscow, B. Vlassievskii 11, 119002 Moscow, Russia}
\curraddr{School of Engineering and Science, Jacobs University Bremen,Campus Ring 1, 28759 Bremen, Germany}
\email{felikson@mpim-bonn.mpg.de}
\thanks{Research of both authors was supported in part by grant 
RFBR 07-01-00390-a and  SNF projects 200020-113199 and 200020-121506/1.}

\author{Pavel Tumarkin}
\address{School of Engineering and Science, Jacobs University Bremen,Campus Ring $1$, 28759 Bremen, Germany}
\email{p.tumarkin@jacobs-university.de}


\title[Subalgebras of hyperbolic Kac-Moody algebras]{Hyperbolic subalgebras of hyperbolic Kac-Moody algebras}

\begin{abstract}
We investigate regular hyperbolic subalgebras of hyperbolic Kac-Moody algebras via their Weyl groups. 
We classify all subgroups relations between Weyl groups of hyperbolic Kac-Moody algebras, and show that for every pair of a group and subgroup their exists at least one corresponding pair of algebra and subalgebra. We also present a finite algorithm classifying all regular hyperbolic subalgebras of hyperbolic Kac-Moody algebras.

\end{abstract}

\maketitle

\setcounter{tocdepth}{1}

\tableofcontents
\addtocounter{section}{-1}
\section{Introduction}
\label{introduction}

In~\cite{Dyn} Dynkin introduced a notion of a {\it regular} subalgebra $\g_1$ of semisimple Lie algebra $\g$. By definition, this is any subalgebra $\h$ of $\g$ generated by root spaces $\g^{\alpha}$ and $\g^{-\alpha}$ for certain roots $\alpha$ of $\g$.  
Dynkin classified regular subalgebras of semisimple Lie algebras in terms of root systems. A {\it root subsystem} of root system $\Delta$ is a root system in $\Delta$ closed with respect to addition.  For his classification Dynkin used the following correspondence: a subalgebra $\g_1\subset \g$ is regular if and only if $\Delta_{\g_1}\subset \Delta_{\g}$ is a root subsystem, where $\Delta_{\g_1}$ and $\Delta_{\g}$ stay for the root systems of $\g_1$ and $\g$ respectively.

Following the definition given in~\cite{Dyn}, it is possible to consider regular subalgebras in 
infinite-dimensional generalizations of semisimple Lie algebras, namely, in Kac-Moody algebras. 
In particular,~\cite{aff} concerns the subalgebras in affine Kac-Moody algebras.
As a next step, it is natural to consider the same question for the case of smallest indefinite Kac-Moody algebras,
i.e. for hyperbolic Kac-Moody algebras.

Some examples of regular hyperbolic subalgebras of hyperbolic Kac-Moody algebras were already studied in several papers.
In~\cite{FN} a series of examples of subalgebras is considered,~\cite{E_10} studies the simply-laced case
and shows that each simply-laced hyperbolic algebra is a subalgebra of $E_{10}$,~\cite{T} and~\cite{T1} study
maximal rank hyperbolic subalgebras and subalgebras of corank 1 in arbitrary hyperbolic Kac-Moody algebras,~\cite{aff} describes all 
semisimple and affine subalgebras of hyperbolic Kac-Moody algebras.

In this paper we consider regular hyperbolic subalgebras of hyperbolic Kac-Moody algebras in full generality.
Following Dynkin~\cite{Dyn}, we reduce the question to the classification of hyperbolic root subsystems in hyperbolic root systems,
which in its turn could be easily reduced to the classification of simplicial subgroups of simplicial groups 
(roughly speaking, a reflection group $G\subset \Isom(\H^n)$ is called {\it simplicial} if there is a reflection preserving isomorphism 
$\phi: G\to G'$ where $G'\subset \Isom(\H^{n'})$, $n'\ge n$, is a reflection group whose  fundamental chamber is a simplex; the precise definition is given in Section~\ref{simpl}). 

Our main results are the classification of simplicial subgroups of simplicial groups, and an algorithm classifying all regular hyperbolic subalgebras of hyperbolic Kac-Moody algebras. We prove also Theorem~\ref{subalg for each subgr} which states that each simplicial subgroup $H$ in arithmetic over $\Q$ simplicial group $G$ corresponds to at least one regular subalgebra  $\h\subset \g$ (more precisely, there exists a  hyperbolic Kac-Moody algebra $\g$ and a regular hyperbolic subalgebra $\h\subset \g$ such that
the Weyl groups of the root systems of $\h$ and $\g$ are isomorphic to  
$H$ and $G$ respectively).

\medskip

The paper is organized as follows. Section~\ref{Preliminaries} serves to recall and introduce most of the notions involved.
In Section~\ref{reduction} we describe how the investigation of hyperbolic subalgebras may be reduced to the investigation 
of simplicial subgroups in simplicial groups. 

In section~\ref{Maximal subgroups} we study basic properties of maximal
(by inclusion) simplicial subgroups which will be extensively used to obtain the classification. We also
prove that in each embedding of hyperbolic algebras $\h\subset \g$ such that the Weyl group of the root system  of $\h$
is a maximal simplicial subgroup of the Weyl group of the root system of $\g$, the algebra $\h$ is always a regular subalgebra of 
$\g$. In particular, this proves the main theorem  (Theorem~\ref{subalg for each subgr})
for the case of maximal subgroup $H\subset G$ with $\rank H<\rank G$. 

Sections~\ref{visual}--\ref{classification} are devoted to classification of simplicial subgroups in hyperbolic simplicial
groups. The subgroups $H\subset G$ of maximal rank ($\rank H=\rank G$) are classified in~\cite{JKRT}, so, we concentrate on
the subgroups of smaller ranks.   
Section~\ref{visual} describes 
an easy way to find almost all  (non-maximal rank) simplicial subgroups in simplicial groups. 
To list all the other subgroups we proceed case by case.
A simplicial rank 2 group (i.e., an infinite dihedral group) is a subgroup of any other simplicial group. 
Furthermore, there are finitely many hyperbolic simplicial groups of rank greater than 3, which implies that we need to check only finitely many pairs $(H,G)$, $3<\rank H<\rank G$.
In Sections~\ref{methods} and~\ref{algorithm} we present different ways to determine 
if a given simplicial group $H$ is a subgroup of the given simplicial group $G$.  
Most general algorithm is described in Section~\ref{algorithm}, however it does not cover several cases, which we treat with the methods explained in Section~\ref{methods}.
The later ones are also in most cases quicker than use of the general algorithm.
The results about subgroups are summarized in Section~\ref{classification} (Theorem~\ref{classification of subgroups}).

Finally, in Section~\ref{back to subalgebras} we apply the classification of subgroups to the investigation of subalgebras.
In particular, we prove Theorem~\ref{subalg for each subgr}), which states that for each subgroup
one can find a regular subalgebra in some algebra.

\subsection*{Acknowledgments}
A large part of the work was performed during the authors' stay at the University of Fribourg, Switzerland. 
The authors are grateful to
the department for the working atmosphere and to R.~Kellerhals for invitation and stimulating discussions.
The first author would also like to thank the Max-Planck Institute for Mathematics in Bonn, where the final version of the paper was written.

\section{Preliminaries}
\label{Preliminaries}
In this section we discuss hyperbolic Kac-Moody algebras and simplicial reflection groups.
We refer to~\cite{Kac} for background on Kac-Moody algebras and to~\cite{29} for details concerning discrete reflection groups. 
In particular, we use the standard notation for finite and affine reflection groups and root systems.

\subsection{Hyperbolic Kac-Moody algebras}
Following Kac~\cite{Kac}, one can construct a Kac-Moody algebra
$\g(A)$ for each generalized Cartan matrix $A$.

A generalized Cartan matrix $A$ is of {\it hyperbolic type} if it is indecomposable symmetrizable of indefinite type and  any proper principle submatrix of $A$ is of finite or affine type. 
In this case the corresponding to $A$ symmetric matrix $B$ is of signature $(n,1)$.

A Kac-Moody algebra $\g(A)$ is called {\it hyperbolic} if the matrix $A$ is of  hyperbolic type. By a {\it hyperbolic root system} we mean a root system of hyperbolic Kac-Moody algebra.

According to Vinberg~\cite{Vinb}, the Weyl group $W$ of the root system $\Delta$ of $\g$ is a group generated by reflections
acting in a hyperbolic space $\H^n$. In case of hyperbolic Kac-Moody algebra $\g$, the fundamental chamber of $W$ is  an $n$-dimensional hyperbolic Coxeter simplex of finite volume, whose dihedral angles are in the set
$\{\frac{\pi}{2},\frac{\pi}{3},\frac{\pi}{4},\frac{\pi}{6}\}$
(zero angle also appears for $n=2$).

\subsection{Simplicial reflection groups}
\label{simpl}
An element $g$ of a group $G\subset  \Isom(\H^n)$  is called {\it parabolic} is $g$ preserves a unique point of $\partial \H^n$.

We say that a reflection group $G\subset \Isom(\H^n)$  is  {\it strictly simplicial} 
if its fundamental chamber  is a finite volume simplex.

A reflection group $G'\subset \Isom(\H^{n'})$  will be called {\it r-isomorphic} to another reflection group $G\subset \Isom(\H^{n})$ 
if there exists an isomorphism $\phi:G'\to G$ preserving the set of reflections and the set of parabolic elements. 
We write $G'\r-equiv G$ for {\rr}-isomorphic groups.

A group $G'\subset \Isom(\H^{n'})$ is {\it simplicial} if it is {\rr}-isomorphic to a strictly simplicial group 
$G\subset \Isom(\H^n)$.

A {\it rank} of a reflection group is the number of reflections in any standard generating set (i.e. the number of facets of the fundamental chamber,
where a {\it facet} is a codimension 1 face of a polytope).
So, for a strictly simplicial group $G\subset \Isom(\H^n)$ we have $\rank G=n+1$,
and for a general simplicial group $G'\r-equiv G$ we obtain $\rank G'=\rank G=n+1$. 

Simplicial groups  (up to an {\rr}-isomorphism) could be conveniently described by Coxeter diagrams.
A {\it Coxeter diagram}  $\Sigma(G)$ of a strictly simplicial group $G\subset \Isom(\H^n)$ is a Coxeter diagram of its fundamental simplex,
namely $\Sigma(G)$ is a graph with $n+1$ vertices,
one vertex for each of facets $f_0,f_1,\dots,f_n$ of a fundamental simplex of $G$. Furthermore,
$i^{th}$ vertex of $\Sigma(G)$ is joined with $j^{th}$ one by an edge labeled by $m_{ij}$ 
if the dihedral angle formed by $f_i$ and $f_j$
is of size $\pi/m_{ij}$. For abuse of notations one uses empty (absent) edges for $m_{ij}=2$, and simple, double, triple and 4-tuple
edges for $m_{ij}=3,4,5,6$ respectively. For the zero angle formed by $f_i$ and $f_j$ in 2-dimensional case
one draws a bold edge. 

The list of Coxeter diagrams of simplicial groups (or, simply, of hyperbolic simplices) is contained in~\cite[Table~4]{29}. 
There are infinitely many strictly simplicial groups in $\H^2$ (but  finitely many with fundamental chambers having angles
in the set $\{0,\pi/2,\pi/3,\pi/4,\pi/5,\pi/6\}$ only) and finitely many in other dimensions. In particular, no strictly simplicial group
acts in $\H^n$, $n>9$. The dihedral angles of fundamental simplices of simplicial groups in $\H^n$, $n\ge 3$ are contained in the set  $\{\pi/2,\pi/3,\pi/4,\pi/5,\pi/6\}$.

We say that a group $G$ is {\it simply-laced} if its Coxeter diagram $\Sigma(G)$ is a simply-laced graph,
i.e. if each edge of $\Sigma(G)$ is either simple or bold.

\begin{remark}
 For each hyperbolic root system $\Delta $ we construct a simplicial group $G(\Delta)$ as a Weyl group of $\Delta$.
Sometimes different root systems may lead to one and the same simplicial group. Namely, $G(\Delta)$ contains all information
about the angles between simple real roots of $\Delta$ but does not know which of the roots are long and which are short.
The root systems are depicted by Dynkin diagrams  which are graphs $\Sigma(G)$ with multiple edges 
enabled by arrows indicating (where arrow from vertex $i$ to $j$ means that $j^{th}$ root is larger than $i^{th}$ one).
Here one exclusion is: Dynkin diagram of dihedral group $\mathbb{G}_2$ of order $12$ is an oriented triple edge, not $4$-tuple as in its Coxeter diagram. 

A simplicial group $G$ corresponds to exactly one hyperbolic root system if  $G$ is simply-laced.
In all the other cases, to pass from a simplicial group to a hyperbolic root system one should specify the appropriate lengths of 
the simple roots (so, two vertices joined by a simple edge correspond to the roots of the same length; if two vertices of 
$\Sigma(G)$ are joined by a double edge then one of the corresponding roots 
is $\sqrt{2}$ times larger than another).
For most of the non-simply-laced groups we may obtain two different root systems by interchanging short roots with long ones
(however, sometimes this gives the same root system up to renumbering of the simple roots, see Fig.~\ref{ex-sh-l}).
\end{remark}

\begin{figure}[!h]
\begin{center}
\epsfig{file=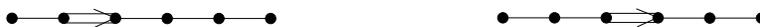,width=0.8\linewidth}
\caption{Changing lengths of the roots may change the root system (left) or may lead to the isomorphic one (right)}
\label{ex-sh-l}
\end{center}
\end{figure}

Notice that a simplicial group $G$ is a Weyl group of a hyperbolic root system if and only if $G$ is arithmetic over $\Q$.
In particular, no pair of roots in hyperbolic root system compose an angle $\pi/5$.

\subsection{Reflection subgroups of simplicial groups}
By a {\it reflection subgroup} of a reflection group we mean a subgroup generated by reflections. By a {\it simplicial subgroup} we mean a reflection subgroup 
which is simplicial.

Let $S=\{r_0,\dots,r_k\}$ be the set of standard (Coxeter) generators of a reflection group $G$,
 i.e. the reflections with respect to the facets of the fundamental chamber of $G\subset \Isom(\H^n)$. 
A reflection subgroup $G_1\subset G$ is called  {\it standard parabolic} if $G_1$ is generated by some collection of the reflections $r_i\in S$. A subgroup $G_1\subset G$ is called  {\it parabolic} if it is conjugated in $G$ to some standard parabolic subgroup.

\begin{remark} 
\label{stabilizers}
If $G$ is cocompact simplicial group then any proper parabolic subgroup is finite  
(since it fixes some subspace of $\H^n$).
Moreover, any maximal proper parabolic subgroup is a finite group stabilizing some point in $\H^n$. 

If a simplicial group $G$ is not cocompact, any finite parabolic subgroup still fixes some subspace of $\H^n$,
while an infinite proper parabolic subgroup stabilizes some point at $\partial \H^n$.

\end{remark}
 
By a {\it mirror} of a reflection  group we mean a hyperplane fixed by some reflection contained in this group.

Let $H\subset G$ be a simplicial subgroup of a strictly simplicial group  $G\subset \Isom(\H^n)$. 
Denote by $\Pi$ a minimal non-trivial $G$-invariant subspace of $\H^n$, and let $H'$ be the restriction of $H$ to $\Pi$.
Clearly, $H\r-equiv H'$ and $H'$ is strictly simplicial reflection subgroup of $\Isom(\Pi)$.  

The following proposition is just to fix several almost evident properties of
simplicial subgroups.

\begin{prop}
\label{properties}
Let $H\subset G$, $H'$ and $\Pi$ be as defined above.

1)  If $G\subset \Isom(\H^n)$ is cocompact then $H'\subset \Isom(\Pi)$ is cocompact.

2) For any parabolic subgroup $H_1'\subset H'$ there is a parabolic subgroup  $G_1\subset G$ 
containing a subgroup {\rr}-isomorphic to $H'_1$.


3) Let $K$ be a simplicial group. If $K$ does not admit 
reflection preserving embedding into $G$, then $K$  has no reflection preserving embeddings into $H$.

\end{prop}

\begin{proof}
1) If $G$ is cocompact then there are no parallel mirrors of reflections of $G$ in $\H^n$. Thus, there are no ones in $H$, and consequently, in $H'$.

2) $G_1$ is the subgroup of $G$ fixing a subspace $X$, where $X\subset \H^n$ is the set fixed by $H'_1$.

3) The statement is evident.

\end{proof}



\section{From subalgebras to subgroups}
\label{reduction}

As before, let $\Delta_{\g}$ be a root system of a hyperbolic Kac-Moody algebra $\g$.
By definition, every regular subalgebra $\h$ of $\g$ has a root system $\Delta_{\h}\subset \Delta_{\g}$ which is closed with respect to addition, 
i.e. satisfies the following condition:
\begin{center}
           if $\alpha, \beta \in \Delta_{\h}$ and $\alpha+\beta \in \Delta_{\g}$, then  $\alpha+\beta \in \Delta_{\h}$.
\end{center}
A root system $\Delta_{\h} \subset \Delta_{\g}$ satisfying the condition above is called a {\it root subsystem} of
$\Delta$. It is easy to see that any
root subsystem  $\Delta_{\h} \subset \Delta_{\g}$ is a root system of a regular subalgebra of $\g$. Therefore,
classifying regular subalgebras of $\g$ is equivalent to classifying root subsystems of  $\Delta_{\g}$.

Furthermore, since $\g$ is a hyperbolic Kac-Moody algebra, a Weyl chamber $C_{\g}$ of $\Delta_{\g}$ is a finite volume 
hyperbolic simplex. The Weyl group  $W_{\g}$ of $\Delta_{\g}$ is generated by reflections with respect to the facets of $C_{\g}$,
so,   $W_{\g}$ is a simplicial reflection group.

Let $\h\subset \g$ be a regular hyperbolic subalgebra and $C_{\h}$ be a Weyl chamber of $\Delta_{\h}$. 
Let $\Pi$ be the minimal $H$-invariant subspace of $\H^n$.
Denote by $C'_{\h}$ the section of $C_{\h}$ by the plane $\Pi$.
Then  $C'_{\h}$ is a finite volume 
hyperbolic simplex in $\Pi$, and  the Weyl group  $W_{\h}$ of $\Delta_{\h}$ is also a simplicial group. Therefore, regular subalgebra $\h\subset \g$ corresponds to a simplicial subgroup $W_{\h}\subset W_{\g}$.

So, the problem of classification of regular hyperbolic subalgebras of $\g$ splits into two steps:

\begin{itemize}
\item {\bf  Step 1:} find the
classification of simplicial reflection subgroups of hyperbolic simplicial reflection groups; 
\item { \bf Step 2:} investigate which subgroups correspond to root subsystems.

\end{itemize}

For the maximal rank subalgebras (i.e. $\rank \h= \rank \g$) these steps are done in~\cite{JKRT} and~\cite{T} respectively.
So, we focus on the case of distinct ranks. Step~1 will be discussed in Sections~\ref{Maximal subgroups}--\ref{classification}.
Step~2 is based on the following three lemmas.

\begin{lemma}[\cite{T}, Lemma~3]
\label{tower}
Let $\Delta_{\g_1}\subset \Delta_{\g_2}$ and $\Delta_{\g_2}\subset \Delta_{\g_3}$ be root subsystems. Then $\Delta_{\g_1}$ is a root subsystem of $\Delta_{\g_3}$.

\end{lemma}

To simplify checking if $\Delta_{\h}$ is a root subsystem of a root system $\Delta_{\g}$, we use the following criterion.

\begin{lemma}
\label{simple roots}
Let $\Pi_1$ be a set of simple roots of a root system $\Delta_1$.
Then $\Delta_1\subset \Delta$ is a root subsystem if and only if 
\begin{equation}
\alpha - \beta \notin \Delta \text{ for all } \alpha,\beta\in \Pi_1.   
\end{equation}
                   
\end{lemma}

The criterion follows immediately from Proposition~\ref{root sub} below.

\begin{prop}[\cite{FN}, Theorem 3.1]
\label{root sub}
Let $\Delta$ be a root system of a Kac-Moody
algebra. Let $\beta_1 ,\dots,\beta_k\in \Delta^{\rm re}$  be positive real roots such that $\beta_i - \beta_j \notin \Delta$ for all
$1 \le i < j \le k$. Let $\Delta' \subset \Delta$ be a minimal root system containing $\beta_1,\dots,\beta_k$. Then
$\Delta'$ is a root subsystem of $\Delta$ (and $\Delta'$ is a root system of a regular subalgebra of $\g$).

\end{prop}

The following lemma shows that in many cases we do not need to check even the simplified criterion.

\begin{lemma}
\label{indecom-angles}
Let  $\Delta_G$ and $\Delta_H\subset \Delta_G$ be hyperbolic root systems
with Weyl groups $G$ and $H\subset G$ respectively.
If each parabolic rank 2 subgroup of $H$ is  parabolic in $G$ then 
 $\Delta_H$ is a root subsystem of $\Delta_G$.
In particular, this is always the case if $\Delta_G$ is simply-laced.

\end{lemma}

\begin{proof}
The condition on parabolic subgroups implies that for any two simple roots $\alpha_1,\alpha_2$ of $\Delta_H$ the vector $\alpha_1-\alpha_2$ is never proportional to
any root of $\Delta_G$ (otherwise the parabolic subgroup of $H$ generated by reflections with respect to $\alpha_1$ and $\alpha_2$ is a proper subgroup of the parabolic subgroup of $G$ containing $\alpha_1$ and   $\alpha_1-\alpha_2$). 
So, Lemma~\ref{simple roots} proves the first statement of the lemma.

Furthermore, if $G$ is simply-laced, then no  rank 2 parabolic subgroup of $G$ contains a proper non-trivial subgroup and the lemma always applies.

\end{proof}

\section{Maximal subgroups}
\label{Maximal subgroups}

A reflection subgroup $H\subset G$ of a reflection group $G$ is called {\it maximal } if it is maximal by inclusion,
i.e.,  for any reflection subgroup $K$ such that $H\subset K \subset G$ we have either
$K=H$ or $K=G$.

To prove the next lemma
we will use the following fact conjectured in~\cite{subgr-h} and proved in more general settings in~\cite{subgr}. 

\begin{prop}[\cite{subgr}, Theorem~1.2] 
\label{subgroup}
 Let $G\subset \Isom (\H^n)$ be a discrete reflection group with a fundamental chamber $F$ of finite volume. 
Let $H\subset G$ be a finite index reflection subgroup.
Then $H$ can not be generated by less than $|F|$ reflections, where $|F|$ is the number of facets of $F$.

\end{prop} 

In particular, if $H$ is a strictly simplicial group in $\H^n$ 
and $K$ is a discrete reflection group containing $H$, then $K$ is also strictly simplicial.

\begin{lemma}
\label{simplicial}
Let $H$ be a maximal simplicial subgroup of a strictly simplicial group $G\subset \Isom(\H^n)$,  $\rank H< \rank G$. 
Let $H_1$ be a finite parabolic subgroup of $H$ and 
let $G_1$ be a minimal parabolic subgroup of $G$ containing $H_1$.
Then  $G_1=H_1$. 

\end{lemma}

\begin{proof}

According to Proposition~\ref{properties}, $G_1 \supset H$ does exist, uniqueness of such a $G_1$ follows immediately.
Suppose that  $H_1\ne G_1$. 
Let $L\subset \H^n$ be a subspace preserved pointwise by $H_1$ (since $H_1$ is finite, $L$ contains some face of a fundamental chamber of $G$). 
Then $G_1$ is a maximal subgroup of $G$ fixing $L$.
Clearly, $H_1\subset G_1$ and $G_1$ is parabolic in $G$ (since $G_1$ is a stabilizer of $L$ in $G$).
Denote by $\Pi$ the minimal non-trivial $H$-invariant subspace of $\H^n$ 
(note that $\Pi$ is orthogonal to $L$, so $\Pi $ is preserved by $G_1$, too).

Now suppose that $H_1\ne G_1$ and
consider  $K=\langle H,G_1 \rangle$. Then $K\ne G$, since $K$ preserves $\Pi$ while $G$ does not.
On the other hand, $H\ne K$, since $H_1\ne G_1$.
If $K$ is simplicial, this contradicts the assumption that  $H\subset G$ is a maximal simplicial subgroup.
So, we are left to show that $K$ is simplicial.

Recall that  $H\subset \Isom(\Pi)$ is a finite covolume simplicial group. 
Since $K$ is a discrete subgroup of  $\Isom(\Pi)$ we see that $[K:H]<\infty$.
So $K$  is a  discrete reflections group  containing a finite index simplicial subgroup.
By Proposition\ref{subgroup},  $K$ is simplicial.

\end{proof}

\begin{remark}
\label{rem-geom sense}
Geometric sense of Lemma~\ref{simplicial} is the following: for any face of fundamental chamber of $H$,
its stabilizer in $H$ coincides with its stabilizer in $G$.
We will heavily use the following reformulation in terms of Coxeter diagrams: 

{\it
For each subdiagram (corresponding to a finite group) of $\Sigma(H)$ there is a subdiagram of the same type in $\Sigma(G)$.

}

\end{remark}

In a simplicial group $H$, any two standard generators are contained in a finite parabolic subgroup
unless $H$ is $\rr$-isomorphic to a group generated by reflections with respect to the sides of non-compact triangle.
For the latter case we will need the following

\begin{lemma}
\label{rk2}
Let $H$ be a maximal simplicial subgroup of a strictly simplicial group $G\subset \Isom(\H^n)$,  $3=\rank H< \rank G$. 
Let $H_1$ be an infinite parabolic subgroup of $H$ and 
let $K_1$ be a maximal infinite dihedral subgroup of $G$ containing $H_1$.
Then  $K_1=H_1$.

\end{lemma}

\begin{proof}
Suppose that $K_1\ne H_1$. 
Following the proof of Lemma~\ref{simplicial}, consider $K=\langle H,K_1 \rangle$.
Then $K\ne H$, since  $K_1\ne H_1$. On the other hand, $K\ne G$ since $K$ preserves the minimal $H$-invariant subspace $\Pi$
 ($\Pi \ne \H^n$ since  $\rank H< \rank G$). By the same reason, $[K:H]<\infty$.
By Proposition~\ref{subgroup},  $K$ is simplicial, which contradicts the assumption that 
$H\subset G$ is a maximal simplicial subgroup.

\end{proof}

Denote by $r_{\alpha}\in W(\Delta)$ the reflection with respect to the root $\alpha \in \Delta $
(where $W(\Delta)$ is the Weyl group of $\Delta$).


\begin{theorem}
\label{subalg for max subgr}
Let $H$ be a maximal simplicial subgroup of a strictly simplicial group $G\subset \Isom(\H^n)$,  $\rank H< \rank G$. 
Let $\Delta_G$ be any root system with $W(\Delta_G)=G$. Then the root system 
$\Delta_H=\{\alpha \in \Delta_G | r_{\alpha}\in H\}$ is a root subsystem
of $\Delta_G$, and $W(\Delta_H)=H$.

\end{theorem}

\begin{proof}
First, suppose that $H$ is not $\rr$-isomorphic to a group generated by reflections with respect to the sides of non-compact triangle.
Then any pair of generating reflections generate a finite parabolic subgroup of $H$. By Lemma~\ref{simplicial}, this dihedral group
is parabolic in $G$. Thus, by Lemma~\ref{indecom-angles}, $\Delta_H$ is a root subsystem of $\Delta_G$. 

Now, let  $H$ be $\rr$-isomorphic to a group generated by reflections with respect to the sides of non-compact triangle.
By Lemma~\ref{simple roots}, it is sufficient to show that $\alpha-\beta\notin \Delta_G$ for all simple roots $\alpha, \beta\in\Delta_H$. This follows from Lemma~\ref{simplicial} if the reflections in $\alpha$ and $\beta$ generate a finite group, and from Lemma~\ref{rk2} otherwise.

The assertion $W(\Delta_H)=H$ is evident.

\end{proof}

The case of equal ranks is treated in~\cite{T}. More precisely,
Theorem~1 of~\cite{T} combined with Lemma~1 of~\cite{T} imply that if  $H\subset G$ is a maximal simplicial 
subgroup and $\rank H=\rank G$, then there is at least one pair of algebras $(\h ,\g)$
such that $\h$ is a regular subalgebra of $\g$ and $\h$ and $\g$ are chosen among the algebras 
with Weyl groups of their root systems coinciding with $H$ and $G$ respectively. 
Combining this with Theorem~\ref{subalg for max subgr} we have

\begin{cor}
\label{subalg for max subgr1}
Let $G\subset \Isom(\H^n)$ be an arithmetic over $\Q$ strictly simplicial group,  
and $H\subset G$ is a maximal 
simplicial subgroup. Then there exist Kac-Moody algebras $\g$ and $\h$ such that 
$G= W_{\g}$, $H= W_{\h}$ and $\h$ is a regular subalgebra of $\g$.

\end{cor}

\section{Visual subgroups}
\label{visual}

In this section we describe a class of subgroups which contains almost all maximal infinite index simplicial subgroups
of simplicial reflection groups. On the other hand, the subgroups of this type  could be easily found 
just by taking a look at the Coxeter diagram of the group $G$.    

We say that a subgroup $H$ of a reflection group $G=\langle r_0,r_1,\dots,r_n \rangle$ is {\it visual} if $H$ is generated by some standard parabolic subgroup $H_1\subset G$  and an additional reflection $r_ir_jr_i$ for some $r_i,r_j$ not belonging to $H_1$.

Let $G$ be a finite covolume reflection group in $\H^n$, and let $F$ be a fundamental chamber of $G$.
Denote by $v_0,v_1,\dots,v_n$ the outward normal vectors of $F$. We normalize $v_k$ so that $(v_k,v_k)=2$.
Then a visual subgroup of $G$ is one generated by reflections with respect to some of $v_k$ and a vector
$$ v_*=v_j-(v_i,v_j)v_i.$$

In particular, if the edge $v_iv_j$ of $\Sigma(G)$ is simple then $ v_*=v_j+v_i$.
The Coxeter diagram $\Sigma(H)$ of the subgroup is then easy to obtain from $\Sigma(G)$ in the following way:
\begin{itemize}
\item shrink  the edge $v_iv_j$  into a new vertex $v_*$;
\item for each vertex $v_k$ joined with both $v_i$ and $v_j$ draw a new edge  $v_kv_*$  
which is bold  if both edges $v_kv_i$ and $v_kv_j$ are simple, and  dotted  otherwise;
\item for each $v_k$ joined with exactly one of $v_i$ and $v_j$ the edge $v_kv_*$ has the same weight as $v_kv_i$ (or $v_kv_j$) had;
\item all the other edges remain intact;
\item take a subdiagram of the obtained diagram if needed.

\end{itemize}

See Fig.~\ref{vis} for an example of visual subgroup.

\begin{figure}[!h]
\begin{center}
\psfrag{i}{$v_i$}
\psfrag{j}{$v_j$}
\psfrag{*}{\small $v_*=v_i+v_j$}
\psfrag{3}{\small $3$}
\psfrag{4}{\small $4$}
\psfrag{5}{\small $5$}
\psfrag{6}{\small $6$}
\psfrag{7}{\small $7$}
\psfrag{8}{\small $8$}
\psfrag{9}{\small $9$}
\psfrag{01}{}
\psfrag{11}{\small $(3,4,4)$}
\epsfig{file=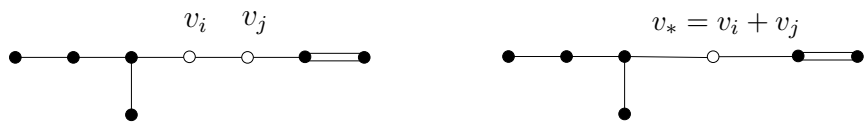,width=0.9\linewidth}
\caption{Example of a visual subgroup $H\subset G$: $\Sigma(G)$ to the left, $\Sigma(H)$ to the right.}
\label{vis}
\end{center}
\end{figure}


\begin{remark} 
As it is easy to see from Tables~\ref{3}--~\ref{6-5}, visual subgroups of corank 1 of simplicial groups turn out to be maximal.
This motivates the following 

\end{remark}

\begin{question} 
Are there visual non-maximal reflection subgroups of corank 1 in generic Coxeter groups?
\end{question}

\section{Subgroup relations: to be or not to be?}
\label{methods}

Given two simplicial groups $G$ and $H$, we need to understand if $H$ can be embedded into $G$ as a reflection subgroup.
For many of the pairs $(G,H)$ we have immediate affirmative answer due to compositions 
$$H=G_0 \subset G_1\subset \dots \subset G_k= G$$
of visual subgroups and subgroups of the maximal rank (i.e. with $\rank G_i=\rank G_{i+1}$).
For all the other pairs $(G,H)$ we need to check subgroup relations.

In this section we describe various methods to check if $H$ can be embedded into $G$. 
We will use these methods to check subgroup relations for almost all pairs $(G,H)$.
The remaining cases are covered by the algorithm presented in Section~\ref{algorithm}.

\subsection{Large balls in upper half-space model}
\label{large balls}

In this section we develop a method  suitable for the case when $H$ is not {\rr}-isomorphic to a cocompact strictly simplicial group.

Our idea is to use geometry of upper half-space model of hyperbolic $n$-space. 
We embed an infinite maximal parabolic subgroup  $H_1\subset H$ into some infinite parabolic subgroup $G_1\subset G$
(see Proposition~\ref{properties}),
assuming that $G_1$ preserves the point at  infinity.
Then the subgroup $H_1\subset G_1$  can be understood as reflection subgroup of a Euclidean reflection group $G_1$ 
(acting on an horosphere centered at infinity), so $H_1$ is easy to describe. 
To understand if $H\subset G$ we need to find (or to prove non-existence of) 
a mirror $m_*$ (i.e. a hemisphere orthogonal to $\partial \H^n$ in the model) forming some prescribed angles with
the mirrors $m_1,\dots,m_k$ corresponding to the reflections generating $H_1$.
Since $m_*$ should intersect (at least at  $\partial \H^n$) all the mirrors bounding the fundamental domain of $H_1$,
the Euclidean radius of the hemisphere is relatively big. Checking hemispheres (mirrors of $G$) of relatively big radius
we either find $m_*$ or prove that it does not exist (we check mirrors of $G$  modulo the group $H_1$, so there are finitely many
of hemispheres in question of radius greater than any given number).

\begin{notation}
We will use the notation introduced in Fig~\ref{but-ex}. In particular, 
notation for the types of hyperbolic simplices is borrowed from~\cite{JKRT-vol}.
A rank 3 group with fundamental triangle with angles $(\pi/q_1,\pi/q_2,\pi/q_3)$ is denoted $(q_1,q_2,q_3)$ (for zero angles we write 0, so that $(0,0,0)$ is the group corresponding to the ideal triangle).

\end{notation}

\begin{figure}[!h]
\begin{center}
\psfrag{0}{}
\psfrag{b1}{\small ${\mathcal F}_2 $}
\psfrag{b2}{\small ${\mathcal F}_3 $}
\psfrag{b3}{\small ${\mathcal F}_4 $}
\psfrag{4}{\small $[3^{1,1,1,1,1}]$}
\psfrag{5}{\small $[4;3^{1,1,1}]$}
\psfrag{6}{\small $(0,2,4)$}
\psfrag{7}{\small $(0,0,2)$}
\psfrag{8}{\small $[(3^2,4^2)]$}
\psfrag{9}{\small $(0,0,3)$}
\psfrag{10}{\small $(0,3,3)$}
\psfrag{11}{\small $(3,4,4)$}
\psfrag{i1}{\small $(0,0,0)$}
\psfrag{i2}{\small $[4^{[4]}]$}
\psfrag{i3}{\small $[3^{[3,3]}]$}
\psfrag{i4}{\small $[(3,6)^{[2]}]$}
\psfrag{i5}{\small $[(3^2,4)^{[2]}$}
\epsfig{file=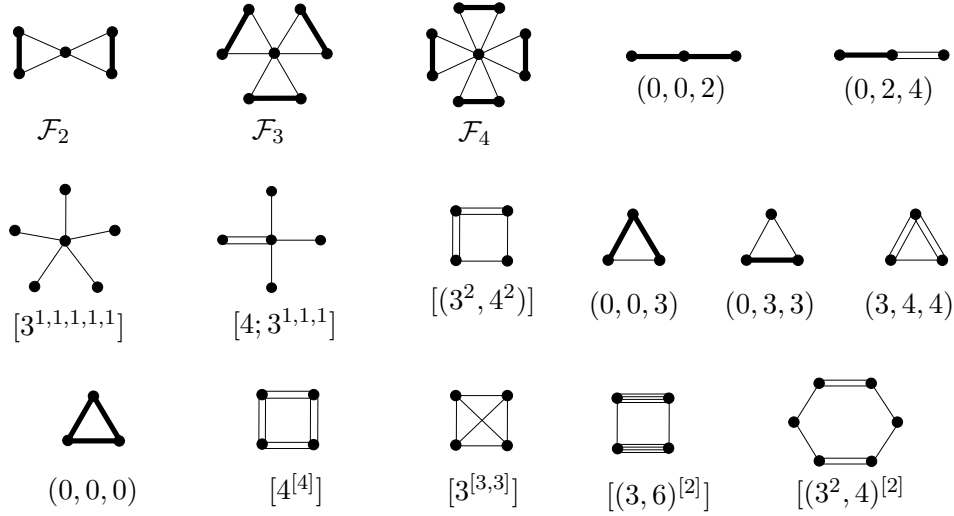,width=0.99\linewidth}
\caption{Notation for some hyperbolic reflection groups. The last line contains the list of all ideal simplices.}
\label{but-ex}
\end{center}
\end{figure}

\begin{example}
\label{ex-id2inid3}
Neither the group $H=(0,0,0)$ nor $H'=(0,3,3)$ is 
  a subgroup of $G=[3^{[3,3]}]$. 
We will prove it for $H$, the statement for $H'$ will follow.

Consider an ideal simplex $F_0$ generating the group $G$.
We place $F_0$ in such a way that one of its vertices $V_{\infty}$ is the point at infinity, and the other three vertices 
form a regular triangle at $\partial \H^3$ (see Fig.~\ref{idreg}(a)). 
We embed the subgroup $H_1$ (which is infinite dihedral group) into the stabilizer of $V_{\infty}$.
Then $H_1$ is generated by reflections with respect to two parallel planes. Denote these planes by
 $\tilde m_1$ and $\tilde m_2$ (now $H_1$ is not a parabolic subgroup
of $G$ but only a subgroup of one).
The third generator of $H$ should be a reflection (inversion in Euclidean sense) with respect to some mirror $m_*$ of the group $G$,
this mirror should be tangent to both   $\tilde m_1$ and $\tilde m_2$.  
It is easy to see that the face $m_0$ of $F_0$ opposite to the vertex  $V_{\infty}$ is the largest hemisphere  
amongst mirrors of $G$.
If we scale the Euclidean distances at  $\partial \H^3$ so that  $m_0$ is a unit hemisphere,
then the Euclidean distance between the closest parallel mirrors of $G$ is $3/2$ (see Fig.~\ref{idreg}(b)).

\begin{figure}[!h]
\begin{center}
\psfrag{a}{(a)}
\psfrag{b}{(b)}
\psfrag{m1}{\scriptsize $\tilde m_1$}
\psfrag{m2}{\scriptsize $\tilde m_2$}
\psfrag{1}{\scriptsize 1}
\psfrag{32}{\scriptsize$\frac{3}{2}$}
\epsfig{file=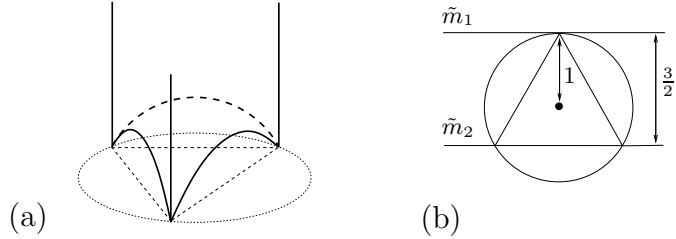,width=0.7\linewidth}
\caption{(a) ideal regular simplex in the upper half-space model, (b) projection onto the boundary.}
\label{idreg}
\end{center}
\end{figure}

In particular, this implies that the planes $\tilde m_1$ and $\tilde m_2$ are the closest parallel planes,
otherwise no hemisphere can touch both of them.

Now we will show that no of the hemispheres except the unit one can have points in common (even at infinity) 
with both   $\tilde m_1$ and $\tilde m_2$.
It is sufficient to show that the radius $R$ of the largest hemisphere which is smaller than $m_0$ 
satisfies $R<3/4$.
We reflect $F_0$ with respect to the mirror $m_0$ and denote by $F_1$ the image of $F_0$.
Notice that the union $F_0\cup F_1$ contains all points of fundamental domain of $H_1$ that lie above 
the horosphere $z=1/2$ in our model,
where $z$ is an axis orthogonal to $\partial \H^n$.
So, besides $m_0$ (and its translates by $H_1$) no hemisphere of the radius greater than $1/2$ can be a mirror of $G$.

A unit hemisphere $m_0$ gives rise to a subgroup of the type $(0,0,3)$.
This implies that no other group of type $(0,p,q)$ is a subgroup of $G$.
In particular, neither $(0,0,0)$ nor $(0,3,3)$ is a subgroup of $G$.

\end{example}

The method of large balls can be also applied for a non-simplicial group $G$. The only restriction is that
 a simplicial group $H$ should have an infinite maximal parabolic subgroup.

\subsection{Method of indefinite coefficients}
\label{indef}

This method is the main tool of~\cite{T1} where the corank 1 hyperbolic subalgebras of hyperbolic
Kac-Moody algebras are investigated.
The method works nicely for all arithmetic over $\Q$ groups $G$.
For this class of groups the problem reduces to some Diophantine equation with integer 
coefficients. For the groups apart from this class the method still sometimes works but it leads to a Diophantine equation with
coefficients and variables in one of the rings $\Z[\sqrt 2 ]$, $\Z[\frac{1+\sqrt 5}{2}]$ and $\Z[\sqrt 2,\frac{1+\sqrt 5}{2}]$. 
We will use the method only for arithmetic over $\Q$ groups $G$.

For each of the $n+1$ facets $f_i$ of the fundamental domain $F$ of $G$
we consider an outward normal vector $v_i$.
We normalize vectors in such a way that the square of length of the shortest ones is equal to 2.
The vectors  corresponding to the ends of a double edge of $\Sigma(G)$ 
would have squares $2$ and $4$,
while vectors  corresponding to the ends of 4-tuple edge of $\Sigma(G)$ 
would have  squares $2$ and $6$.
Then any vector of type $gv_j$, $g\in G$, is an integer combination of $v_i$'s.

We embed a maximal parabolic subgroup $H_1\subset H$ into some parabolic subgroup $G_1\subset G$
writing down the explicit expressions of vectors of generating reflections $u_1,\dots,u_k$ in terms of $v_0,v_1,\dots,v_n$.
Denote by $u_*$ the vector corresponding to the missed generating reflection of $H$
and write $$u_*=\sum\limits_{i=0}^{n} k_iv_i.$$
Vector $u_*$ satisfies several equations. Namely, the value of each of
$(u_*,u_i)$ (where $i=1, \dots,k$) can be read from the Coxeter diagram $\Sigma(H)$.
In this way we get $k$ linear equations with integer coefficients with respect to $k_i$.
Another one (quadratic) equation  can be obtained in the following way: $\Sigma(G)$ encodes the length of $u_*$, so
$(u_*,u_*)$ is known.

It turns out that in most cases either the system has no solutions (usually it is enough to
consider these equations modulo 2),
or one can find among the solutions one corresponding 
to a vector $v_*$ which could be obtained from some of $v_i$ by the action of $G$.

\begin{example}
\label{ideal tetrahedron}
A group $H=(0,0,3)$
is not a subgroup of $G=[3^{1,1,1,1,1}]$ (see Fig.~\ref{star33333d}).

Let $v_0$ be a vector corresponding to the valence 5 vertex of $\Sigma(G)$ and let 
$v_1,\dots,v_5$ be the vectors corresponding to the remaining vertices, set $(v_i,v_i)=2$.  
The only finite maximal parabolic subgroup of $H$ can be embedded into $G$ in a unique up to a symmetry way:
$u_1=v_0$, $u_2=v_1$. Let $u_*=\sum\limits_{i=0}^{5} k_iv_i$.
Then we have
\begin{center}
\begin{tabular}{l}
$-2=(u_*,v_0)=2k_0-\sum\limits_{i=1}^5 k_i,$ \\
$-2=(u_*,v_1)=2k_1-k_0,$                     \\ 
$\phantom{-}2=(u_*,u_*)=2\sum\limits_{k=0}^5 k_i^2-2k_0\sum\limits_{i=1}^5 k_i.$ 
\end{tabular}
\end{center}

From the first two of these equations we find 
$$k_0=2k_1+2 \text{ \quad  and \quad } k_5=3k_1-k_2-k_3-k_4+6.$$
Eliminating $k_0$ and $k_5$ from the third equation, we get
$$1=4(k_1+1)^2+\sum\limits_{i=1}^4 k_i^2+(3k_1-k_2-k_3-k_4+6)^2-2(k_1+1)(4k_1+6).$$ 
This equation modulo 2 is equivalent to
$$1=\sum\limits_{i=1}^4 k_i^2+\left(\sum\limits_{i=1}^4 k_i\right)^2$$
which is impossible.
 
\end{example}

\begin{figure}[!h]
\begin{center}
\psfrag{a}{(a)}
\psfrag{b}{(b)}
\psfrag{0}{\scriptsize $v_0$}
\psfrag{1}{\scriptsize $v_1$}
\psfrag{2}{\scriptsize $v_2$}
\psfrag{3}{\scriptsize $v_3$}
\psfrag{4}{\scriptsize $v_4$}
\psfrag{5}{\scriptsize $v_5$}
\psfrag{*}{\scriptsize $u_*$}
\epsfig{file=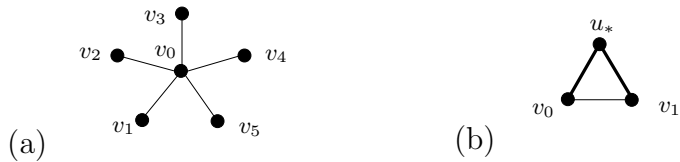,width=0.7\linewidth}
\caption{(a): Dynkin diagram for $G=[3^{1,1,1,1,1}]$;  (b) a unique possible embeddings of $H=(0,0,3)$.}
\label{star33333d}
\end{center}
\end{figure}

\medskip 
\noindent
Next example demonstrates this method in the non-simply-laced setting.
 
\begin{example} 
\label{ex3undef}
A group $H=(0,2,4)$ 
is not a subgroup of $G=[4,3^{1,1,1}]$  (see Fig.~\ref{star3334d}). 

For the group $G$ we take a root system $\Delta_G$  shown in Fig.~\ref{star3334d}(a) 
and fix a numeration of the vertices. Then we have two possibilities for the embedding of $H$ 
(see Fig.~\ref{star3334d}(b) and (c)). We consider the first of them, the second one is treated in the same way.
We have \\
\phantom{qwqwqwqwq}
$ (u_*,v_1)=-4=-2k_0+4k_1-2k_2-2k_3-2k_4,$\\ 
\phantom{qwqwqwqwq}
$ (u_*,v_4)=0=-2k_1+2k_4,$ \\
which implies that $$-4=(u_*,v_1+v_4)=2(-k_0+k_1-k_2-k_3),$$
so, $k_0+k_1+k_2+k_3$ (as well as $k_0+k_2+k_3+k_4$) is even.
On the other hand,
$$4=(u_*,u_*)=2k_4^2+ 4(k_0+k_1+k_2+k_3)^2-4k_1(k_0+k_2+k_3+k_4).$$
Dividing this equation by 2 and rewriting modulo 4, we obtain
$$2\equiv k_4^2 \mod 4,$$
which is impossible.

\end{example}

\begin{figure}[!h]
\begin{center}
\psfrag{a}{(a)}
\psfrag{b}{(b)}
\psfrag{c}{(c)}
\psfrag{v0}{\scriptsize $v_0$}
\psfrag{v1}{\scriptsize $v_1$}
\psfrag{v2}{\scriptsize $v_2$}
\psfrag{v3}{\scriptsize $v_3$}
\psfrag{v4}{\scriptsize $v_4$}
\psfrag{u}{\scriptsize $u_*$}
\epsfig{file=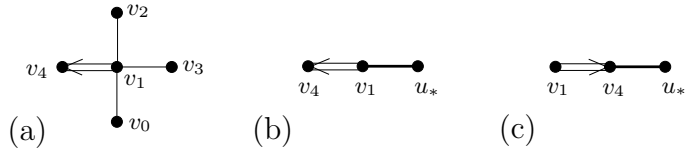,width=0.7\linewidth}
\caption{(a): Dynkin diagram for $G=[4,3^{1,1,1}]$;  (b),(c): two possible embeddings of $H=(0,2,4)$.}
\label{star3334d}
\end{center}
\end{figure}

\medskip

The method of indefinite coefficients can be applied for investigation of simplicial subgroup of a non-simplicial group $G$.

The results of the following example will be used in Section~\ref{section short-long}.
Here we use the method for the case when the fundamental domain of $G$ 
is a pyramid over a cube.

\begin{example} 
\label{f2}
The groups $G'={\mathcal F}_2$ and $G={\mathcal F}_3$ 
contain no subgroup $H=(0,0,0)$ (see Fig~\ref{but-num} for the notation).

We prove the statement for the group $G$, then by Proposition~\ref{properties}, the statement follows for $G'\subset G$.

Denote the vertices of $\Sigma(G)$ as on  Fig~\ref{but-num}. Let $H_1$ be any maximal parabolic subgroup of $H$.
Suppose that $H\subset G$.
By Lemma~\ref{rk2}, we may assume that $H_1$ is generated by reflections with respect to $v_1$ and $v_2$ 
(up to an $\rr$-isomorphism of $G$). Therefore, $H$ is embedded as shown on Fig~\ref{but-num} on the right, 
where $u_*=\sum\limits_{i=0}^{6}k_iv_i$.
The computation similar to one in Example~\ref{ideal tetrahedron} shows that such an embedding does not exist.

\end{example}

\begin{figure}[!h]
\begin{center}
\psfrag{a}{(a)}
\psfrag{b}{(b)}
\psfrag{c}{(c)}
\psfrag{d}{(d)}
\psfrag{0}{\scriptsize $v_0$}
\psfrag{1}{\scriptsize $v_1$}
\psfrag{2}{\scriptsize $v_2$}
\psfrag{3}{\scriptsize $v_3$}
\psfrag{4}{\scriptsize $v_4$}
\psfrag{5}{\scriptsize $v_5$}
\psfrag{6}{\scriptsize $v_6$}
\psfrag{7}{\scriptsize $v_7$}
\psfrag{8}{\scriptsize $v_8$}
\psfrag{u}{\scriptsize $u_*$}
\epsfig{file=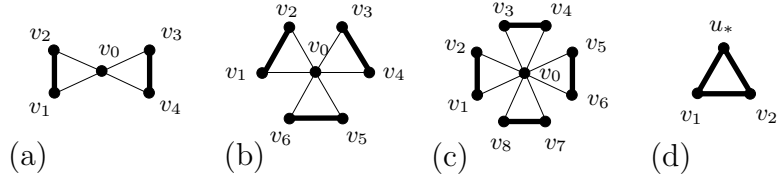,width=0.8\linewidth}
\caption{(a)-(c) Groups ${\mathcal F}_2$, ${\mathcal F}_3$ and ${\mathcal F}_4 $;  (d) non-existing embedding of $H=(0,0,0)$.}
\label{but-num}
\end{center}
\end{figure}

\begin{remark}
The group $G={\mathcal F}_4$ contains a subgroup  $H=(0,0,0)$.

Embeddings of the type discussed in Example~\ref{f2} still do not exist, however there is another embedding
 due to the fact that the fundamental chamber of $G$ has different types of ideal vertices:
one ideal vertex of ``cubical'' combinatorial type (with faces corresponding to $v_1,\dots,v_8$) and two simplicial ideal vertices 
of type $\widetilde D_4$ (corresponding to $v_0,v_1,v_3,v_5,v_7$ and $v_0,v_2,v_4,v_6,v_8$ respectively).
The embedding of $H_1$ into  any of subgroups of  type $\widetilde D_4$ can be easily lifted to embedding of $H\subset G$.

\end{remark}

\subsection{Short and long subgroups}
\label{section short-long}
A simple method described in this section allows to reduce the problem to consideration of pairs $(G,H)$ where both
$G$ and $H$ are simply-laced. 
The idea is to split the problem into two easier ones.
This method works efficiently for many pairs $(G,H)$. 

Suppose that  $G$ is a Weyl group of a non-simply-laced root system.
In most cases this implies that there are at least two possibilities for such a root system 
(we can reverse all arrows in Dynkin diagram). For each of the Weyl groups we fix one of the possible root systems, namely we take one of the  root systems $\Delta$ shown in Table~\ref{short-long}. Notice that each of them has roots of exactly two lengths.

We say that a reflection $r$ in $G$ is {\it short} if $r$ is a reflection with respect to a short root of $\Delta$.
Similarly, we define {\it long} reflections.
A subgroup $G_{s}\subset G$ generated by all short reflections is the {\it short subgroup} of $G$.  
Similarly, a {\it long subgroup} is the subgroup  $G_l\subset G$ generated by all long reflections.

\medskip

Suppose that $H$ is a simply-laced group. Then  $H\subset G$ if and only if $H\subset G_s$ or $H\subset G_l$.
If $H$ is not simply-laced, then $H\subset G$ implies 
$$\text{either } H_s\subset G_s \text{ and } H_l\subset G_l,  \text{\quad  or \quad} H_s\subset G_l \text{ and } H_l\subset G_s.$$ 
On the other hand, converse is not true: if the condition  above holds, we can say nothing about the existence of the subgroup $\rr$-isomorphic to $H$ in $G$.

\medskip

Now, we will study short and long subgroups of simplicial groups.
Denote by ${\mathcal D}$ the Dynkin diagram of $\Delta$. Let ${\mathcal D}_{s}$ be the subdiagram of ${\mathcal D}$ spanned by the nodes corresponding to
short roots, and $\mathcal D_l$ be the subdiagram of $\mathcal D$  spanned by the long roots.

\begin{lemma}
\label{sh-l}
The group $G_{s}$ is a finite index subgroup of a simplicial group $G$ unless $\mathcal D_{l}$ is an affine diagram.
Similarly, $G_l$ is a finite index subgroup of $G$ unless $\mathcal D_{s}$ is an affine diagram.

\end{lemma}

\begin{proof}
We prove the first statement of the lemma. Since the proof is invariant under reversing of all arrows in $D$,
the second one follows by symmetry.

Let $G_1$ be a parabolic subgroup of $G$ generated by the reflections with respect to simple long roots.
Denote by $L$ the intersection of all the mirrors of reflections of $G_1$.
Since $\mathcal D_{l}$ is not an affine diagram, $G_1$ is finite, so $L$ is not an ideal point, denote $\dim L = k$.
Let $F$ be a fundamental chamber of $G$ having a $k$-dimensional face in $L$.
Consider the union of the fundamental chambers of $G$ which are in the $G_1$-orbit of $F$, denote it by $P$: 
$$P=\bigcup\limits_{g\in G_1} gF.$$

Then $P$ is a Coxeter polytope. Indeed, each dihedral angle of $P$ is either a dihedral angle of some of $gF$ or 
a dihedral angle of a union of two images $g_1F$ and $g_2F$ of $F$. The latter is possible only if
 $g_1F$ and $g_2F$ are attached to each other by a facet corresponding to a long root, while the facets of the angle correspond 
to short roots.
Therefore, this dihedral angle is equal to a doubled angle of $F$ which is either $\pi/4$ or $\pi/6$ (as an angle between roots 
of different lengths).  
So, each angle of $P$ is of form $\pi/m$ where $m\in \Z$, and $P$ is a Coxeter polytope.

Clearly, the reflection group $G_P$ generated by reflections in facets of $P$ is a subgroup of $G_{s}$. 
Furthermore, every mirror of $G$ intersecting the interior of $P$ is a mirror of a long reflection.
Thus, $G_P$ coincides with $G_{s}$, and $[G:G_{s}]$ is the order of the stabilizer of $L$ in $G$.

\end{proof}

In Table~\ref{short-long} we list short and long subgroups of arithmetic over $\Q$ simplicial groups.


\begin{table}[!h]
\begin{center}
\caption{Short and long subgroups. By $|G|$ we mean the order of a group $G$.} 
\label{short-long}
\begin{tabular}{|c|c|c c|c c|}
\hline
& $\Delta$ & $G_{s}$ & $[G:G_{s}]$& $G_l$ & $[G:G_{l}]$\\
\hline
1&\epsfig{file=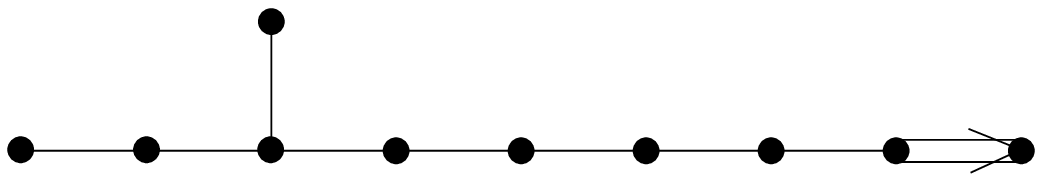,width=0.2\linewidth}
& 
\begin{tabular}{c}
right-angled\\
$\infty$-cell\\
\end{tabular}
& $\infty$ &
\epsfig{file=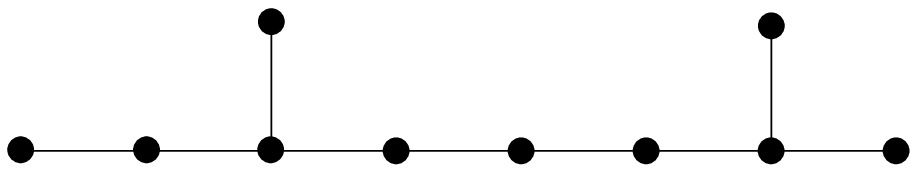,width=0.18\linewidth}
& $2$\\

\hline
2& \epsfig{file=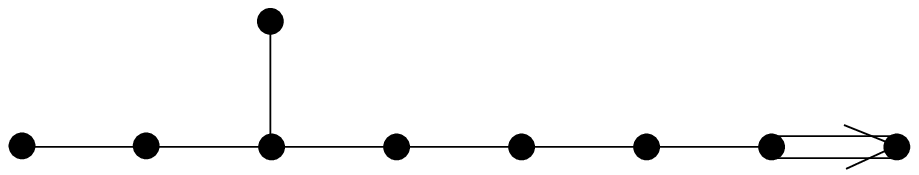,width=0.2\linewidth}
& 
\begin{tabular}{c}
right-angled\\
$240$-cell\\
\end{tabular}
& $|E_8|$ &
\epsfig{file=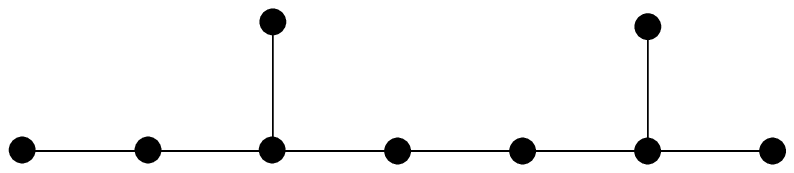,width=0.18\linewidth}
& $2$\\

\hline
3&\epsfig{file=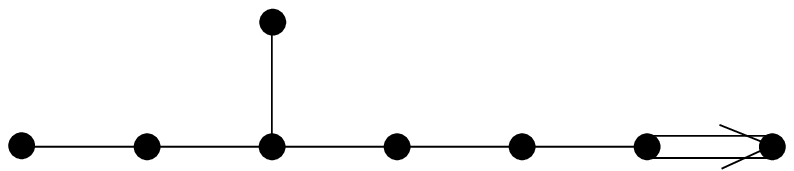,width=0.2\linewidth}
& 
\begin{tabular}{c}
right-angled\\
$56$-cell\\
\end{tabular}
& $|E_7|$ &
\epsfig{file=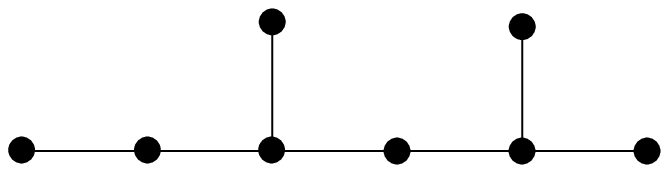,width=0.18\linewidth}
& $2$\\

\hline
4&\epsfig{file=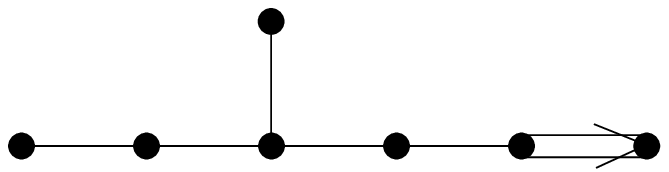,width=0.2\linewidth}
& 
\begin{tabular}{c}
right-angled\\
$27$-cell\\
\end{tabular}
& $|E_6|$ &
\epsfig{file=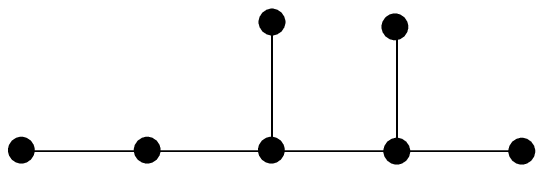,width=0.18\linewidth}
& $2$\\

\hline
5&\begin{tabular}{c}\epsfig{file=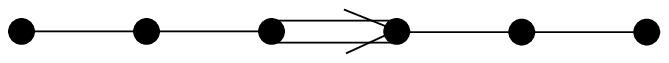,width=0.2\linewidth}\end{tabular}
& \begin{tabular}{c}\epsfig{file=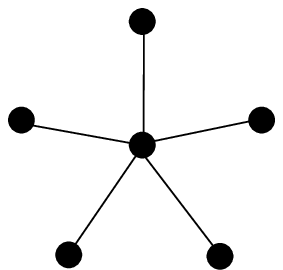,width=0.18\linewidth}\end{tabular}
& $24$ &
\begin{tabular}{c}\epsfig{file=./pic/star1.eps,width=0.18\linewidth}\end{tabular}
& $24$\\

\hline
6& \begin{tabular}{c}\epsfig{file=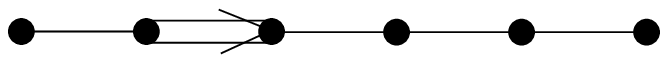,width=0.2\linewidth}\end{tabular}
& \begin{tabular}{c}\epsfig{file=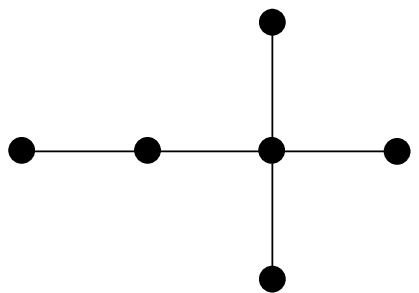,width=0.18\linewidth}\end{tabular}
& $6$ &
\begin{tabular}{c}\epsfig{file=./pic/star1.eps,width=0.18\linewidth}\end{tabular}
& $120$\\

\hline
7& \begin{tabular}{c}\raisebox{-0pt}[20pt][0pt]{\epsfig{file=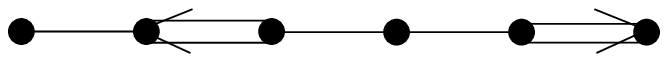,width=0.2\linewidth}}\end{tabular}
& \begin{tabular}{c}\raisebox{-0pt}[20pt][0pt]{\epsfig{file=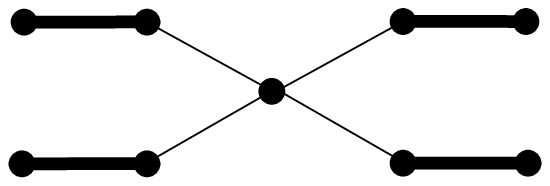,width=0.18\linewidth}}\end{tabular}
& $24$ &
\begin{tabular}{c}\raisebox{-0pt}[20pt][0pt]{\epsfig{file=./pic/star1.eps,width=0.18\linewidth}}\end{tabular}
& $12$\\

\hline
8& \begin{tabular}{c}\epsfig{file=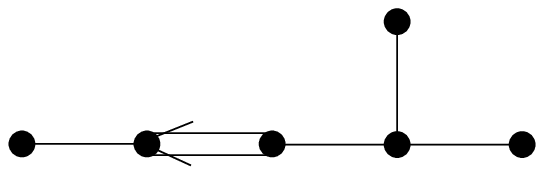,width=0.2\linewidth}\end{tabular}
&  \begin{tabular}{c}\epsfig{file=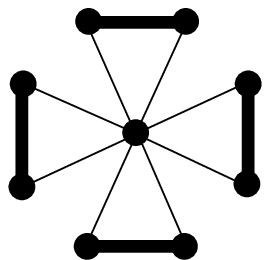,width=0.18\linewidth}\end{tabular}
& $192$ &
\begin{tabular}{c}\epsfig{file=./pic/star1.eps,width=0.18\linewidth}\end{tabular}
& $6$\\

\hline
9&\begin{tabular}{c}\epsfig{file=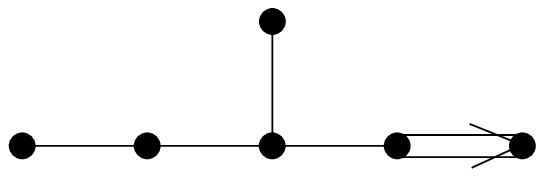,width=0.2\linewidth}\end{tabular}
& 
\begin{tabular}{c}
right-angled \\
16-cell
\end{tabular}
& $1920$ &
\begin{tabular}{c}\epsfig{file=./pic/star_.eps,width=0.18\linewidth}\end{tabular}
& $2$\\

\hline
10& \begin{tabular}{c}\epsfig{file=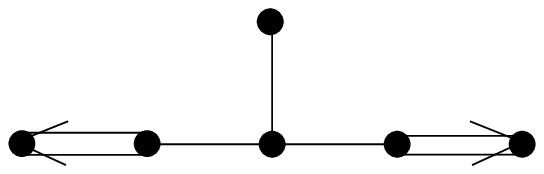,width=0.2\linewidth}\end{tabular}
& 
\begin{tabular}{c}
right-angled \\
16-cell
\end{tabular}
& $192$ &
\begin{tabular}{c}\epsfig{file=./pic/star1.eps,width=0.18\linewidth}\end{tabular}
& $4$\\

\hline
11&\begin{tabular}{c}\raisebox{-0pt}[25pt][0pt]{\epsfig{file=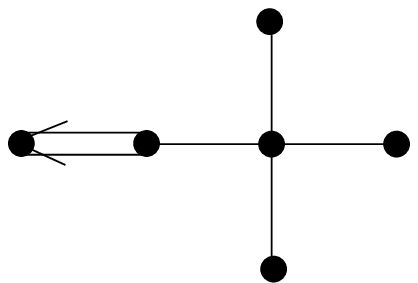,width=0.2\linewidth}}\end{tabular}
& 
\begin{tabular}{c}
right-angled\\
$\infty$-cell\\
\end{tabular}
& $\infty$ & \begin{tabular}{c}\epsfig{file=./pic/star1.eps,width=0.18\linewidth}\end{tabular}
& $2$\\

\hline
12& \begin{tabular}{c}\raisebox{-0pt}[20pt][0pt]{\epsfig{file=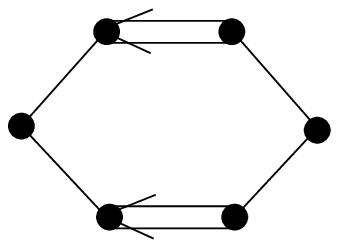,width=0.2\linewidth}}\end{tabular}
& \begin{tabular}{c}\raisebox{-0pt}[20pt][0pt]{\epsfig{file=./pic/butterfly4_.eps,width=0.18\linewidth}}\end{tabular}
& $24$ &
\begin{tabular}{c}\raisebox{-0pt}[20pt][0pt]{\epsfig{file=./pic/butterfly4_.eps,width=0.18\linewidth}}\end{tabular}
& $24$\\

\hline
13&\begin{tabular}{c}\raisebox{-0pt}[20pt][0pt]{\epsfig{file=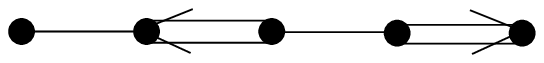,width=0.2\linewidth}}\end{tabular}
& \begin{tabular}{c}\raisebox{-0pt}[20pt][0pt]{\epsfig{file=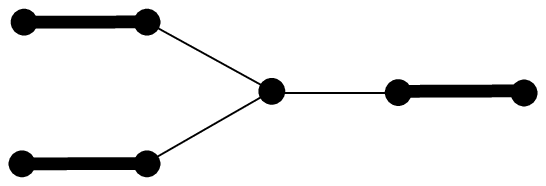,width=0.18\linewidth}}\end{tabular}
& $6$ &
\begin{tabular}{c}\raisebox{-0pt}[20pt][0pt]{\epsfig{file=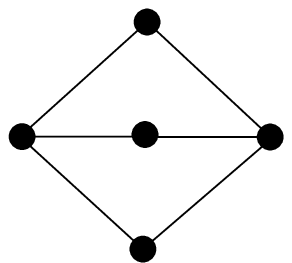,width=0.18\linewidth}}\end{tabular}
& $12$\\

\hline
14&\begin{tabular}{c}\raisebox{-0pt}[20pt][0pt]{\epsfig{file=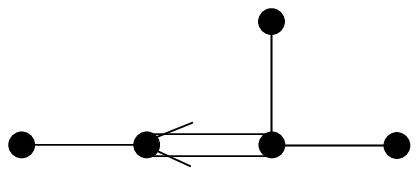,width=0.2\linewidth}}\end{tabular}
& \begin{tabular}{c}\raisebox{-0pt}[20pt][0pt]{\epsfig{file=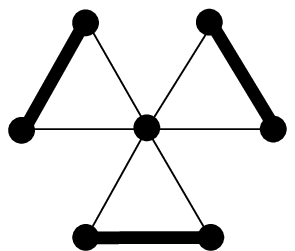,width=0.18\linewidth}}\end{tabular}
& $24$ &
\begin{tabular}{c}\raisebox{-0pt}[20pt][0pt]{\epsfig{file=./pic/l13.eps,width=0.18\linewidth}}\end{tabular}
& $6$\\

\hline
15& \begin{tabular}{c}\epsfig{file=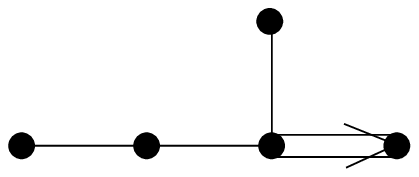,width=0.2\linewidth}\end{tabular}
& 
\begin{tabular}{c}
right-angled \\
10-cell
\end{tabular}
& $120$ &
\begin{tabular}{c}\epsfig{file=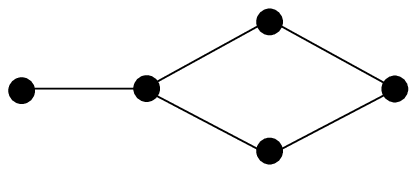,width=0.18\linewidth}\end{tabular}
& $2$\\

\hline
16&\begin{tabular}{c}\epsfig{file=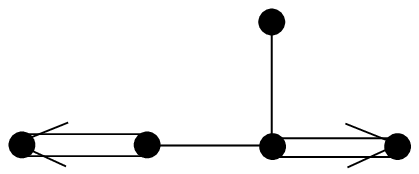,width=0.2\linewidth}\end{tabular}
& 
\begin{tabular}{c}
right-angled \\
10-cell
\end{tabular}
& $24$ &
\begin{tabular}{c}\epsfig{file=./pic/l13.eps,width=0.18\linewidth}\end{tabular}
& $4$\\

\hline
17&\begin{tabular}{c}\epsfig{file=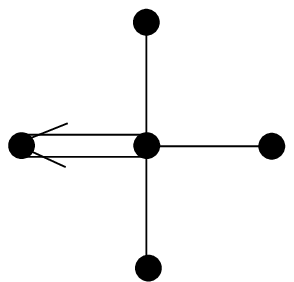,width=0.2\linewidth}\end{tabular}
& 
\begin{tabular}{c}
right-angled \\
24-cell
\end{tabular}
& $192$ &
\begin{tabular}{c}\epsfig{file=./pic/l13.eps,width=0.18\linewidth}\end{tabular}
& $2$\\

\hline
18&\begin{tabular}{c}\raisebox{-0pt}[20pt][0pt]{\epsfig{file=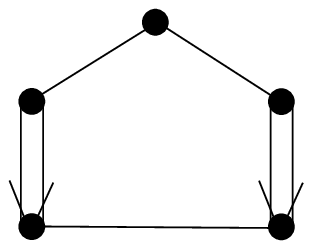,width=0.2\linewidth}}\end{tabular}
& \begin{tabular}{c}\raisebox{-0pt}[20pt][0pt]{\epsfig{file=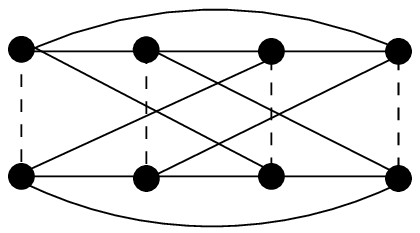,width=0.18\linewidth}}\end{tabular}
& $24$ &
\begin{tabular}{c}\raisebox{-0pt}[20pt][0pt]{\epsfig{file=./pic/butterfly3_.eps,width=0.18\linewidth}}\end{tabular}
& $6$\\

\hline
19&\begin{tabular}{c}\epsfig{file=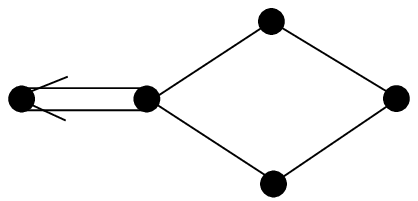,width=0.2\linewidth}\end{tabular}
& 
\begin{tabular}{c}
right-angled\\
$\infty$-cell\\
\end{tabular}
& $\infty$ &
\begin{tabular}{c}\epsfig{file=./pic/l13.eps,width=0.18\linewidth}\end{tabular}
& $2$\\

\hline
\end{tabular}
\end{center}
\end{table}

\pagebreak
\clearpage
\addtocounter{table}{-1}

\begin{table}[!h]
\begin{center}
\caption{ Cont.}
\label{short-long2}
\begin{tabular}{|c|c|c c|c c|}
\hline
& $\Delta$ & $G_{s}$ & $[G:G_{s}]$& $G_l$ & $[G:G_{l}]$\\
\hline
20&\begin{tabular}{c}\epsfig{file=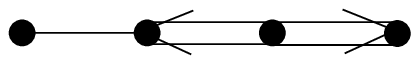,width=0.12\linewidth}\end{tabular}
& \begin{tabular}{c}\epsfig{file=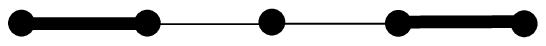,width=0.14\linewidth}\end{tabular}
& $2$ &
\begin{tabular}{c}\raisebox{-0pt}[25pt][0pt]{\epsfig{file=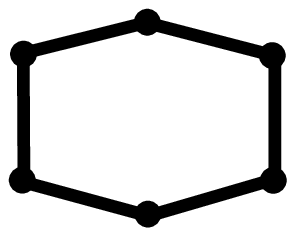,width=0.14\linewidth}}\end{tabular}
& $12$\\

\hline
21&\begin{tabular}{c}\epsfig{file=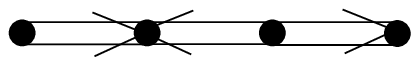,width=0.12\linewidth}\end{tabular}
& \begin{tabular}{c}\raisebox{-0pt}[25pt][0pt]{\epsfig{file=./pic/bi.eps,width=0.14\linewidth}}\end{tabular}
& $4$ &
\begin{tabular}{c}\raisebox{-0pt}[25pt][0pt]{\epsfig{file=./pic/bi.eps,width=0.14\linewidth}}\end{tabular}
& $4$\\

\hline
22&\begin{tabular}{c}\raisebox{-0pt}[20pt][0pt]{\epsfig{file=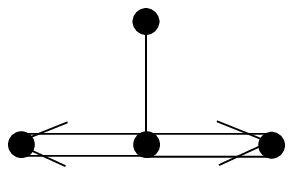,width=0.12\linewidth}}\end{tabular}
& \begin{tabular}{c}\raisebox{-0pt}[25pt][0pt]{\epsfig{file=./pic/bi.eps,width=0.14\linewidth}}\end{tabular}
& $6$ &
\begin{tabular}{c}\raisebox{-0pt}[16pt][0pt]{\epsfig{file=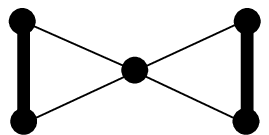,width=0.14\linewidth}}\end{tabular}
& $4$\\

\hline
23&\begin{tabular}{c}\raisebox{-0pt}[20pt][0pt]{\epsfig{file=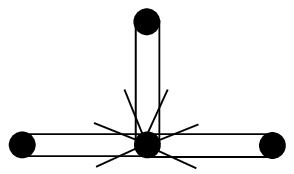,width=0.12\linewidth}}\end{tabular}
& 
\begin{tabular}{c}
right-angled \\
octahedron
\end{tabular}
& $8$ &
\begin{tabular}{c}\raisebox{-0pt}[25pt][0pt]{\epsfig{file=./pic/bi.eps,width=0.14\linewidth}}\end{tabular}
& $2$\\

\hline
24&\begin{tabular}{c}\raisebox{-0pt}[20pt][0pt]{\epsfig{file=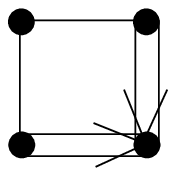,width=0.12\linewidth}}\end{tabular}
& 
\begin{tabular}{c}
right-angled \\
12-cell
\end{tabular}
& $24$ &
\begin{tabular}{c}\raisebox{-0pt}[16pt][0pt]{\epsfig{file=./pic/butterfly2_.eps,width=0.14\linewidth}}\end{tabular}
& $2$\\

\hline
25&\begin{tabular}{c}\raisebox{-0pt}[20pt][0pt]{\epsfig{file=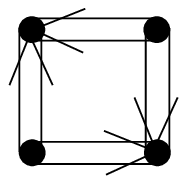,width=0.12\linewidth}}\end{tabular}
& 
\begin{tabular}{c}
right-angled \\
octahedron
\end{tabular}
& $4$ &
\begin{tabular}{c}
right-angled \\
octahedron
\end{tabular}
& $4$\\

\hline
26&\begin{tabular}{c}\raisebox{-0pt}[20pt][0pt]{\epsfig{file=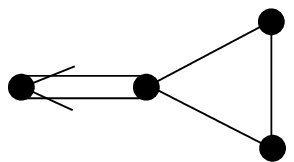,width=0.12\linewidth}}\end{tabular}
& 
\begin{tabular}{c}
right-angled\\
$\infty$-cell\\
\end{tabular}
& $\infty$ &
\begin{tabular}{c}\raisebox{-0pt}[20pt][0pt]{\epsfig{file=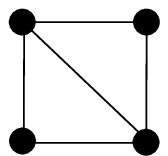,width=0.2\linewidth}}\end{tabular}
& $2$\\

\hline
27&\begin{tabular}{c}\raisebox{-0pt}[20pt][0pt]{\epsfig{file=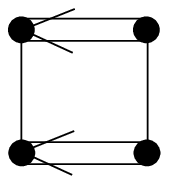,width=0.12\linewidth}}\end{tabular}
& \begin{tabular}{c}\raisebox{-0pt}[20pt][0pt]{\epsfig{file=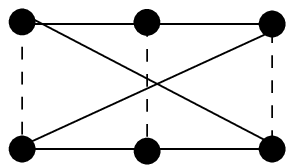,width=0.14\linewidth}}\end{tabular}
& $6$ &
\begin{tabular}{c}\raisebox{-0pt}[20pt][0pt]{\epsfig{file=./pic/cube3.eps,width=0.14\linewidth}}\end{tabular}
& $6$\\

\hline
28&\begin{tabular}{c}\epsfig{file=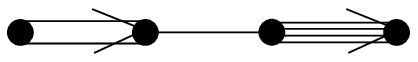,width=0.12\linewidth}\end{tabular}
& 
\begin{tabular}{c}
regular cube, \\
angles $=\pi/3$
\end{tabular}
& $48$ &
 \begin{tabular}{c}
right-angled\\
$\infty$-cell\\
\end{tabular}
 & $\infty$\\
                                                   
\hline
29&\begin{tabular}{c}\raisebox{-0pt}[20pt][0pt]{\epsfig{file=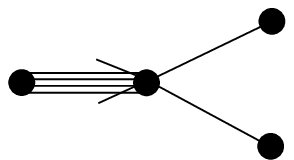,width=0.12\linewidth}}\end{tabular}
& \begin{tabular}{c}\raisebox{-0pt}[20pt][0pt]{\epsfig{file=./pic/semiid3_.eps,width=0.14\linewidth}}\end{tabular}
& $2$ &
\begin{tabular}{c}
regular cube, \\
angles $=\pi/3$
\end{tabular}
& $24$\\
                                                                                                                        
\hline
30&\begin{tabular}{c}\raisebox{-0pt}[20pt][0pt]{\epsfig{file=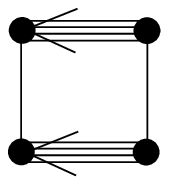,width=0.12\linewidth}}\end{tabular}
& 
\begin{tabular}{c}
regular cube, \\
angles $=\pi/3$
\end{tabular}
& $6$ &
\begin{tabular}{c}
regular cube, \\
angles $=\pi/3$
\end{tabular}
& $6$\\
                                                                                                 
\hline
31&\begin{tabular}{c}\epsfig{file=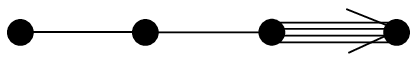,width=0.12\linewidth}\end{tabular}
& \begin{tabular}{c}\raisebox{-0pt}[20pt][0pt]{\epsfig{file=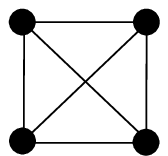,width=0.14\linewidth}}\end{tabular}
& $24$ &
\begin{tabular}{c}\raisebox{-0pt}[20pt][0pt]{\epsfig{file=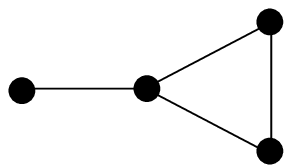,width=0.14\linewidth}}\end{tabular}
& $2$\\
                                              
\hline
32&\begin{tabular}{c}\epsfig{file=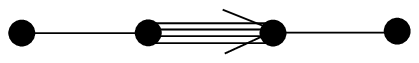,width=0.12\linewidth}\end{tabular}
& \begin{tabular}{c}\raisebox{-0pt}[20pt][0pt]{\epsfig{file=./pic/ideal3_.eps,width=0.14\linewidth}}\end{tabular}
& $6$ &
\begin{tabular}{c}\raisebox{-0pt}[20pt][0pt]{\epsfig{file=./pic/ideal3_.eps,width=0.14\linewidth}}\end{tabular}
& $6$\\
\hline

33&\begin{tabular}{c}\raisebox{-0pt}[20pt][0pt]{\epsfig{file=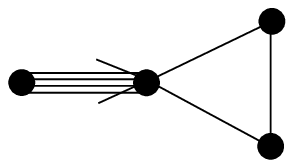,width=0.12\linewidth}}\end{tabular}
& \begin{tabular}{c}\raisebox{-0pt}[20pt][0pt]{\epsfig{file=./pic/ideal3_.eps,width=0.14\linewidth}}\end{tabular}
& $2$ &
 \begin{tabular}{c}
right-angled\\
$\infty$-cell\\
\end{tabular}
 & $\infty$\\

\hline
34&\begin{tabular}{c}\epsfig{file=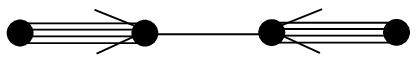,width=0.12\linewidth}\end{tabular}
& \begin{tabular}{c}\raisebox{-0pt}[20pt][0pt]{\epsfig{file=./pic/ideal3_.eps,width=0.14\linewidth}}\end{tabular}
& $4$ &
\begin{tabular}{c}\raisebox{-0pt}[20pt][0pt]{\epsfig{file=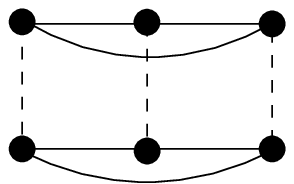,width=0.14\linewidth}}\end{tabular}
& $6$\\

\hline
35&\begin{tabular}{c}\raisebox{-0pt}[20pt][0pt]{\epsfig{file=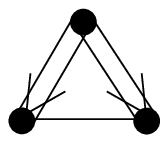,width=0.12\linewidth}}\end{tabular}
& \begin{tabular}{c}\raisebox{-0pt}[20pt][0pt]{\epsfig{file=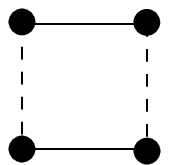,width=0.14\linewidth}}\end{tabular}
& $2$ &
 \begin{tabular}{c}
right-angled \\
hexagon\\
\end{tabular}
&
$6$\\

\hline
36&\begin{tabular}{c}\raisebox{-0pt}[20pt][0pt]{\epsfig{file=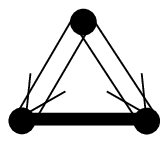,width=0.12\linewidth}}\end{tabular}
& \begin{tabular}{c}\raisebox{-0pt}[20pt][0pt]{\epsfig{file=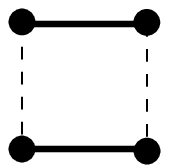,width=0.14\linewidth}}\end{tabular}
& $2$ &
 \begin{tabular}{c}
right-angled\\
$\infty$-cell\\
\end{tabular}
 & $\infty$\\

\hline
37&\begin{tabular}{c}\raisebox{-0pt}[20pt][0pt]{\epsfig{file=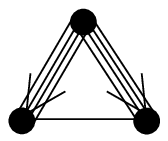,width=0.12\linewidth}}\end{tabular}
& \begin{tabular}{c}\raisebox{-0pt}[20pt][0pt]{\epsfig{file=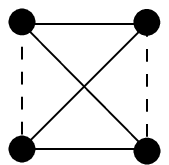,width=0.14\linewidth}}\end{tabular}
& $4$ &
 \begin{tabular}{c}
hexagon, \\
angles $=\pi/3$
\end{tabular}
& $6$\\

\hline
38&\begin{tabular}{c}\raisebox{-0pt}[20pt][0pt]{\epsfig{file=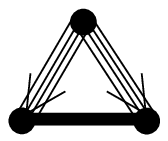,width=0.12\linewidth}}\end{tabular}
& \begin{tabular}{c}\raisebox{-0pt}[20pt][0pt]{\epsfig{file=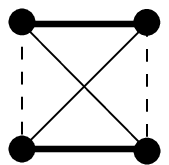,width=0.14\linewidth}}\end{tabular}
& $2$ &
 \begin{tabular}{c}
right-angled\\
$\infty$-cell\\
\end{tabular}
 & $\infty$\\

\hline
\end{tabular}
\end{center}
\end{table}

\clearpage


\begin{example}
The group $G=(3,4,4)$ 
contains no subgroup $H=(0,0,0)$. 

It is sufficient to prove that $H$ is not a subgroup of the larger group $G'=[(3^2,4^2)]$ (see Fig.~\ref{but-ex} for the notation),
where $G\subset G'$ is a visual embedding described in Section~\ref{visual}.
Two roots corresponding to parallel mirrors either are of the same length, or one of them is twice larger than another.
Since $G'$ contains reflections of two lengths only, we conclude that all reflections of $H$ are of the same length.
Therefore, if $H$ can be embedded to $G'$, it can be embedded either into $G'_{s}$ or into $G'_{l}$.
 
The group $G'$ corresponds to two root systems, we choose one whose Dynkin diagram is shown in row 24 of Table~\ref{short-long}.
Then $G'_{s}$ is a group generated by reflections of a compact 3-polytope (see Table~\ref{short-long}),
so it can not contain $H$ (see part 1 of Proposition~\ref{properties}).
As it is shown in Example~\ref{f2}, 
$G'_l$  does not contain $H$ either. 
Hence, $H$ is not a subgroup of $G'$. 

\end{example}

%
%

\section{Algorithm}
\label{algorithm}

Given two simplicial groups $G$ and $H$ ($\rank H<\rank G$) we need to determine if $H$ can be embedded into $G$ as a maximal simplicial subgroup.
In this section we develop an algorithm which either finds an embedding or proves that one does not exist.
The algorithm works unless $H$ is {\rr}-isomorphic to a strictly simplicial group whose fundamental chamber is an ideal simplex. 
In other words, the algorithm works if at least one maximal parabolic subgroup of $H$ is finite.
 
Throughout this section we suppose that at least one of the maximal parabolic subgroups of $H$ is finite.
Denote by $H_1$ this finite maximal parabolic subgroup.
Since we are looking for maximal subgroups,   Lemma~\ref{simplicial} implies that $H_1$ is also a parabolic subgroup of $G$.

The same algorithm works well for the case $\rank H=\rank G$
(but it requires some minor changes for the case when $H_1$ is not a parabolic subgroup of $G$).

\begin{remark}
The algorithm produces the list of embedding containing all maximal ones. 
However, some of the embeddings produced by the algorithm might not be maximal. 

\end{remark}

\bigskip
\bigskip
\pagebreak
\subsection{Notation}

\begin{center}
\begin{tabular}{ll}
$\rho(A,B)$  & distance between the points $A$ and $B$;\\
$\rho(A,\Pi)$ &distance between a point $A$ and a hyperplane $\Pi$;\\
$G$ &finite covolume simplicial group in $\H^n$;\\
$H\subset G$ &simplicial subgroup;\\
$r_0,r_1,\dots,r_n$ & standard generating reflections of $G$;\\
$m_0,m_1,\dots,m_n$ & mirrors (fixed planes) of $r_0,r_1,\dots,r_n$ respectively;\\
$H_1\subset H$ & finite maximal parabolic subgroup of $H$;  \\
$H_1=\langle r_1,\dots, r_k\rangle$ & reflections generating $H_1$, $k\le n$;\\
$H=\langle r_*,H_1\rangle$ 
&  reflections generating $H$;\\
$m_*$ & mirror of reflection $r_*$;\\
$L=\bigcap\limits_{1\le i\le k} m_i$ & a subspace of $\H^n$ stabilized pointwise by $H_1$ \\
&(a point if $\rank H=\rank G$);\\
$\Pi $ & smallest $H$-invariant plane ($\dim \Pi=k$); \\
$h$ & the common perpendicular of $L$ and $m_*$\\ & (or a perpendicular from $L$ to $m_*$ if $L$ is a point);\\
$d=\rho(L,m_*)$ & distance between $L$ and $m_*$.

\end{tabular} 
\end{center}

\subsection{Example: case of equal ranks}
\label{equal ranks}
 Let $G\subset \Isom(\H^2)$ be a group generated by the reflections in sides of a triangle with angles 
$(\frac{\pi}{2},\frac{\pi}{3},\frac{\pi}{7})$.
It might contain as a subgroup a group $H$ generated by the reflections in sides of a triangle with angles 
$(\frac{\pi}{3},\frac{\pi}{3},\frac{\pi}{7})$.
Indeed, the covolume of $G$ is 
$$\pi-\frac{\pi}{2}-\frac{\pi}{3}-\frac{\pi}{7}=\frac{\pi}{42},$$
while the covolume of $H$ is exactly 8 times larger:
$$\pi-\frac{\pi}{3}-\frac{\pi}{3}-\frac{\pi}{7}=\frac{8\pi}{42}.$$
So, $H$ might be an index 8 subgroup of $G$. We prove that it is not the case.

Suppose that $H$ is embedded into $G$ so that $P$ is a fundamental chamber of $H$.
Take angle $\angle O$ (see Fig.~\ref{237}(a)) of size $\frac{\pi}{3}$ of $P$ and identify it with the angle of the same size
of some fundamental domain $F_0$ of $G$ 
(this corresponds to taking a parabolic subgroup $H_1=G_1$ such that the sides of this angle are $m_1$
and $m_2$ in  our notation).
We start to reflect $F_0$, so that the copies $F_i$ of $F_0$ fill the angle $\angle O$  (see Fig.~\ref{237}(a)).
Suppose that we know that $F_i\subset P$ for some $i$ (this is the case for $F_0$).  
For each mirror $m_j$ (obtained as a line containing a side of $F_i$) we check that it does not compose the angles
$\frac{\pi}{3}$ and $\frac{\pi}{7}$ with $m_1$ and $m_2$ respectively
(of course, we also need to check the inverse pair $(\frac{\pi}{7},\frac{\pi}{3})$).
Then $m_j$ contains no side of $P$ (unless $m_j\in\{ m_1,m_2\}$), which implies that the reflection image $F_{i'}=r_{m_j}F_i$ of $F_i$ 
with respect to $m_j$ is also contained in $P$.

When we run through 8 copies of $F_0$ and obtain more the 8 copies of $F_0$ contained in $P$,
 we stop and conclude that the subgroup does not exists (otherwise $P$ should be tiled by exactly 8 copies of $F_0$ since $[G:H]=8$).

%

\begin{figure}[!h]
\begin{center}
\psfrag{a}{(a)}
\psfrag{b}{(b)}
\psfrag{F0}{\tiny $F_0$}
\psfrag{F1}{\tiny $F_1$}
\psfrag{F2}{\tiny $F_2$}
\psfrag{F3}{\tiny $F_3$}
\psfrag{F4}{\tiny $F_4$}
\psfrag{F5}{\tiny $F_5$}
\psfrag{F6}{\tiny $F_6$}
\psfrag{F7}{\tiny $F_7$}
\psfrag{O}{$O$}
\psfrag{p7}{\scriptsize $\frac{\pi}{7} $}
\psfrag{p3}{\scriptsize $\frac{\pi}{3}$}
\psfrag{d}{$d$}
\psfrag{m}{$m_*$}
\psfrag{m0}{\tiny $m_0$}
\psfrag{m1}{$m_1$}
\psfrag{m2}{$m_2$}
\epsfig{file=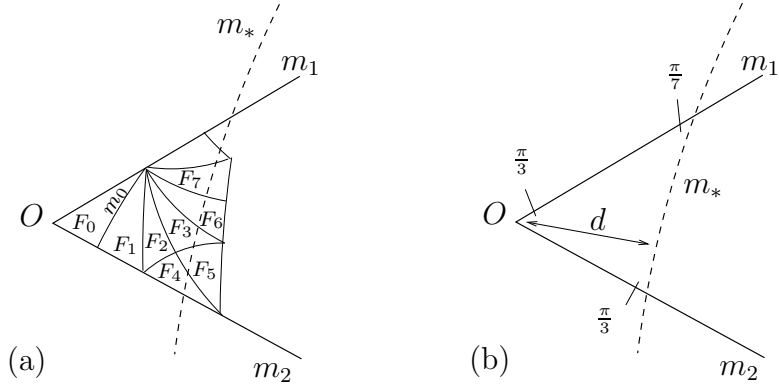,width=0.8\linewidth}
\caption{Checking that  $(\frac{\pi}{3},\frac{\pi}{3},\frac{\pi}{7})$ is not a subgroup of 
$(\frac{\pi}{2},\frac{\pi}{3},\frac{\pi}{7})$.}
\label{237}
\end{center}
\end{figure}

Unfortunately, this reasoning is difficult to generalize to the case of distinct ranks, since in this case the only possible index 
is infinite, and consequently it is not quite clear where to stop. 
On the other hand, we can find another reason to stop.
Indeed, if there exists a mirror $m_*$ which forms angles $\frac{\pi}{3}$ and $\frac{\pi}{7}$ with $m_1$ and $m_2$
(or reverse), then the distance $d$ from the vertex $O$ of the angle to the line $m_*$ can be easily found 
as the distance from the vertex to the opposite side in the triangle  $(\frac{\pi}{3},\frac{\pi}{3},\frac{\pi}{7})$.
Thus, instead of running through the first 8 copies of $F_0$ we need to run through all the copies of $F_0$ 
being inside the closed ball of radius $d$ centered in $O$.

\medskip
\noindent
In case of distinct ranks we still can embed $H_1$ into a parabolic subgroup of $G$ 
and start the reflection process described above.
In practice, if $H\subset G$ then the process will quickly find the missed mirror $m_*$
(and thus, the existence of the subgroup will be established).
However, if $H$ is not a subgroup of $G$ we do not know exact place where to stop the search.  
In what follows we define the compact balls to which we may restrict our search.

\subsection{Cocompact groups.}
Suppose that $G$ is a cocompact group. We will show that in this case the situation does not differ too much from one described in
2-dimensional example above.

Suppose that $H\subset G$ is embedded as a reflection subgroup, denote by $P$ the fundamental chamber of $H$
bounded by $m_1,\dots, m_k$ and $m_*$. 
Consider the $k$-plane
$\Pi $  invariant with respect to $H$. 
The section of $P$ by $\Pi$ is a $k$-simplex (denote it by $P'$).
Since any line perpendicular to both $L=\bigcap\limits_{i=1}^km_i$ and $m_*$ belongs to $\Pi$ (and coincides with the altitude of the simplex $P'$
going from the vertex $X=L\cap \Pi$ to the opposite side $m_*\cap \Pi$),
the distance $d=\rho(m_*,L)$ between $m_*$ and $L$ is equal to the distance from the vertex $X=L\cap \Pi$ of $P'$ to the opposite facet of $P'$
(see Fig.~\ref{cocompact}(a)).
This restricts us to a ``cylinder'' of radius $d$ centered in $L$ (i.e. to the set of points of $\H^n$
lying at distance at most $d$ from $L$). Unfortunately, this cylinder is still of infinite volume, so it is intersected by infinitely many fundamental domains of $G$. 

\begin{figure}[!h]
\begin{center}
\psfrag{a}{(a)}
\psfrag{b}{(b)}
\psfrag{L}{$L$}
\psfrag{F}{\scriptsize $F_0$}
\psfrag{O}{$O$}
\psfrag{X}{$X$}
\psfrag{H}{$\Pi$}
\psfrag{d}{$d$}
\psfrag{ma}{\tiny \phantom{D}$\le D_F$}
\psfrag{m}{$m_*$}
\psfrag{m1}{$m_1$}
\psfrag{m2}{$m_2$}
\epsfig{file=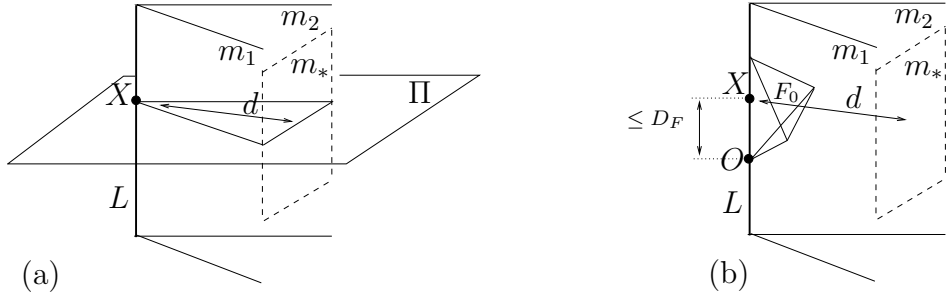,width=0.99\linewidth}
\caption{Geometry of cocompact case}
\label{cocompact}
\end{center}
\end{figure}

To overcome this infinite freedom of moving along $L$ 
we choose the fundamental simplex $F_0$ of $G$ in such a way that 
$X\in F_0$ (where $X=L\cap \Pi$) and $F_0$ has $m_1,\dots, m_k$ as facets, see Fig.~\ref{cocompact}(b). 
Notice that $X$ may not be a vertex of $F_0$.
Let $O$ be any vertex of $F_0$ contained in the subspace $L$ (such a vertex does exist since $L$ contains an $(n-k)$-face of $F_0$).
Denote by $D_{F}=D_{F_0}$ the diameter of $F_0$ (we define diameter as a maximal distance in $\H^n$ between two points of $F_0$). 
Here $D_{F}<\infty$ since $G$ is cocompact.
Then $\rho(O,X)\le D_{F}$, so we have
$$
\rho(O,m_*)\le \rho(O,X)+\rho(X,m_*)\le D_{F} + d.   
$$
Therefore, we may restrict the search to the ball of radius $ D_{F} + d$ centered in $O$.

\subsection{Finite covolume groups.}

Now we generalize the method above to the non-cocompact case.
Now the fundamental domain $F$ is not compact, so the reasoning above meets two obstacles:

1) The face $L\cap F$ may contain ideal vertices only;

2) Diameter of $F$ is infinite. 

Thus, we need to redefine the point $O$ and the value $D_{F}$.
For that we use the following proposition:

\begin{prop}[\cite{Ap}, Theorem~4.28]
\label{independence}
Let $G\subset \Isom(\H^n)$ be a discrete group. 
If the stabilizer $G_x$ of a point $x\in \partial \H^n$ contains a parabolic translation $t$ then the horoball 
$B_{x}=\{y \ | \ \rho(ty,y)\le 1 \}$ 
centered at $x$ satisfies $B_x\cap g (B_x)=\emptyset$ for any $g\in G\setminus G_x$.

\end{prop}

In other words, Proposition~\ref{independence} defines a horoball on which only $G_x$ acts non-trivially. 

Let $V$ be a point of $\partial \H^n$ such that the stabilizer $G_V$ of $V$ in $G$ contains a parabolic translation.
Let $\mathcal O$ be a horoball centered in $V$. The group $G_V$ acts on the horosphere $\partial \mathcal O$. 
Take a point $Z\in \partial \mathcal O$ on the horosphere.  
We call a parabolic element $t\in G_V$ a {\it minimal parabolic translation} in $G_V$ 
 if for any parabolic element $t'\in G_V$ we have $\rho(Z,tZ)\le\rho(Z,t'Z)$. 
We say that $\mathcal O$ is a {\it standard horoball} if $\rho(Z,tZ)=1$ for a minimal parabolic translation $t$.

Now we will use standard horoballs to define $O$ and $D_{F}$.
For each ideal vertex $V_i$ of $F_0$ we consider a standard horoball ${\mathcal O}_i$ centered at this vertex.
Denote by $F_{c}$ the ``compact part'' of $F_0$: 
$$F_{c}:=F_0\setminus \bigcup\limits_{i} {\mathcal O}_i.$$
We define $O$ as any point of $F_{c}\cap L$ and $D_{F}$ as a diameter of  $F_{c}$.

The domain $F_{c}$ is not convex (due to its horospherical ``facets''). 
Nevertheless, we will show that its diameter $D_{F}$ can be computed as a maximum distance between its vertices.
Let $\{W_1,\dots,W_s\}$ be the set of all vertices of  $F_{c}$: 
if $F$ has $q$ ideal vertices, then there are $n-q+1$ ``usual'' vertices and $qn$ vertices obtained as an intersection 
of a horosphere with an edge of $F$ (notice that by Proposition~\ref{independence} horoballs ${\mathcal O}_i$ and ${\mathcal O}_j$
have no points in common for $i\ne j$).

\begin{lemma}
$D_F=\max\limits_{0<i,j\le s} \rho(W_i,W_j)$.

\end{lemma}

\begin{proof}
The inequality $D_F\ge \max\limits_{0<i,j\le s} \rho(W_i,W_j)$ is evident.
To prove that  $D_F\le \max\limits_{0<i,j\le s} \rho(W_i,W_j)$, 
take convex hull $F'_c$ of  $\{W_1,\dots,W_s\}$. The polytope $F'_c$ contains $F_c$, thus the diameter of $F_c$ does not exceed
 the diameter of $F'_c$, while the latter coincides with the right hand side of the inequality.

\end{proof}

\begin{lemma} 
$\rho(O,m_*)\le 2d+D_F$, where $D_F$ is the diameter of $F_{c}$.
\end{lemma}  

\begin{proof}
Consider a common perpendicular $h$ of $L$ and $m_*$, denote its ends by $X\in L$ and $Y\in m_*$.
By assumption, the length of $h$ is $d=\rho(X,Y)$, and we need to find an upper bound for $\rho(O,Y)$ (and then use $\rho(O,m_*)\le \rho(O,Y)+\rho(X,Y)$).

If $X\in F_{c}$, then  $$\rho(O,Y)\le \rho(O,X)+\rho(X,Y)\le D_{F}+d$$
and the lemma follows.

Suppose that  $X\notin F_{c}$. Then $X\in {\mathcal O}_i$ for some standard horoball $ {\mathcal O}_i$.
Denote by $X_{{\mathcal O}_i}$ the point of   $\partial {\mathcal O}_i$ closest to $X$.
It is clear that $X_{{\mathcal O}_i}\in  F_{c}$ (see Fig.~\ref{radius}). 
By Proposition~\ref{independence}, the mirror $m_*$ does not intersect the standard horoball  ${\mathcal O}_i$.
So, $Y$ lies outside of  ${\mathcal O}_i$, while $X$ lies inside.
Therefore, $$d=\rho(X,Y)\ge \rho(X,X_{{\mathcal O}_i}).$$ 
In addition, $X_{{\mathcal O}_i}\in  F_{c}$, which implies
$\rho(O,X_{{\mathcal O}_i})\le D_F$. So, 
$$
\rho(O,Y)\le \rho(O,X_{{\mathcal O}_i})+\rho(X_{{\mathcal O}_i},X)+\rho(X,Y)\le D_F+d+d=D_F+2d, 
$$
and the lemma is proved.

\begin{figure}[!h]
\begin{center}
\psfrag{L}{$L$}
\psfrag{X}{$X$}
\psfrag{Xo}{$X_{{\mathcal O}_i}$}
\psfrag{Y}{$Y$}
\psfrag{m}{$m_*$}
\psfrag{Oi}{${\mathcal O}_i$}
\psfrag{O}{$O$}
\psfrag{d}{$\partial \H^n$}
\epsfig{file=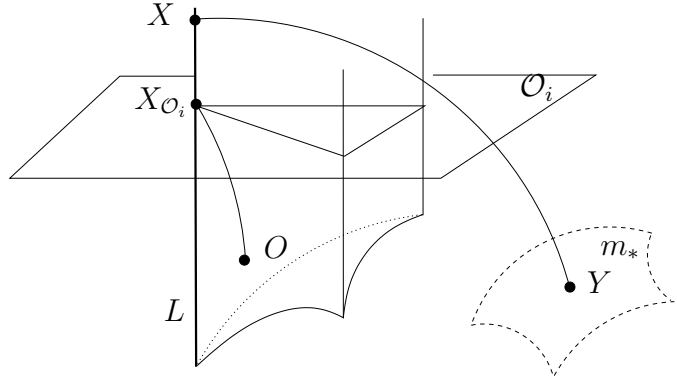,width=0.7\linewidth}
\caption{$\rho(O,m_*)\le 2d+D_F$, proof.}
\label{radius}
\end{center}
\end{figure}

\end{proof}

\subsection{Description of the algorithm}
The reasoning above can be summarized in the following algorithm.

Let $H$ and $G$ be two simplicial groups such that each maximal finite parabolic subgroup of $H$ is a parabolic subgroup of $G$.
Suppose that the fundamental domain of $H$ is not an ideal simplex.
Then to determine if $H$ embeds as a maximal simplicial subgroup into $G$ it is sufficient to do the following:

\begin{itemize}
\item[{\bf A.}] Fix any finite maximal parabolic subgroup $H_1$ of $H$ and embed it into $G$ as a parabolic subgroup.

\item[{\bf B.}] Find $d$ and $D_F$,
then find the radius $R$ of the ball. If $G$ is cocompact then $R$ is given by the following formula:
$$
R=d +D_F. 
$$
If $G$ is not cocompact, then 
$$
R=2d+D_F.
$$

\item[{\bf C.}]
Choose a fundamental simplex $F_0$ of the group $G$ with a face contained in $L$ (the face is of the dimension $\rank H -1$).
Take any point $O\in L\cap F_{c}$.
Reflecting simplex $F$ with respect to its facets, find all fundamental simplices intersecting the ball $B_R(O)$ of radius $R$
centered in $O$.
For each mirror of $G$ intersecting the ball $B_R(O)$ check if the reflection in this mirror together with $H_1$ generate a group 
{\rr}-isomorphic to $H$. In other words, check if this mirror composes the same angles with $m_1,...,m_k$ as it should.

\end{itemize}

Since the ball  $B_R(O)$ is compact and $G$ is a discrete group, there are finitely many copies  of $F$
intersecting  $B_R(O)$, so all of them can be listed in finitely many steps. 

\begin{remark}
While proving non-existence of an embedding (or looking for all possible embeddings), in Step A we need to embed $H_1$ into $G$ as a parabolic subgroup in all possible ways (up to inner automorphism of $G$).

\end{remark}

\noindent
Now we describe in more details steps B and C.

\medskip 
\noindent
{\bf B1. Finding $d$:} 
Let $P$ be a fundamental domain of $H'$. The parabolic subgroup $H_1$ is a stabilizer of some vertex $V$ of $P$.
Let $f$ be a facet of $P$ opposite to the vertex $V$. Then $d=\rho(V,f)$.  

\medskip 
\noindent
{\bf B2. Finding $D_F$:} 
If $G$ is cocompact, then $D_F$ is just a diameter of $F$ (where $F$ is a fundamental domain of $G$).
If $G$ is not cocompact, then for each ideal vertex of $F$ we find explicitly the standard horosphere ${\mathcal O}_i$
and compute $D_F$ as a diameter of $F_c=F\setminus \bigcup\limits_i {\mathcal O}_i$.

\medskip 
\noindent
{\bf C. Running through all mirrors of $G$ intersecting $B_R(O)$:} 
We introduce distance on the set of fundamental simplices of $G$ as the distance on the graph dual to the tessellation of $\H^n$ by fundamental chambers. Namely, two fundamental simplices of $G$ are at distance 1 if they have a common facet.
Two fundamental simplices $F$ and $F'$ are at distance $k$ if there is a sequence of fundamental simplices
$F=F_0,F_1,\dots,F_{k-1},F_k=F'$ such that $F_i$ is at distance 1 from $F_{i+1}$ (for all $0\le i <k$)
and there is no shorter sequence satisfying this conditions.

Each mirror intersecting the closed ball $B_R(O)$ is a facet of some fundamental simplex (which also intersects the ball).
So, we look through all the fundamental simplices starting from the simplices on distance one from $F$,
then check simplices on distance 2, and so one.
We stop either when the subgroup is found or if the ball is exhausted.

\begin{remark}
The algorithm works unless  $H$ is isomorphic to a group generated by any of the five ideal simplices, see the lower row of 
Fig.~\ref{but-ex} for their Coxeter diagrams and notation.  
However, in view of Lemma~\ref{simplicial} no of the simplices $[(3,6)^{[2]}]$ and $[(3^2,4)^{[2]}]$  
can generate a subgroup of a simplicial group of higher rank 
(they have parabolic subgroups of the types $G_2$ and $F_4$ respectively).
So, the algorithm can be applied unless $H$ is $(0,0,0)$, $[4^{[4]}]$ or $[3^{[3,3]}]$. 

\end{remark}


\section{Classification of subgroups}
\label{classification}

In this section we list simplicial subgroups $H\subset G$ of simplicial reflection groups of rank greater than 3.
Section~\ref{ar} is devoted to subgroups of arithmetic over $\Q$ groups, in
 Section~\ref{non-ar} we list subgroups of the remaining ones.
We start with the subgroups of the simply-laced groups, then 
proceed by decreasing of the smallest dihedral angle of the fundamental chamber of $G$ while it is greater or equal to $\pi/6$. 
A dihedral angle smaller than $\pi/6$ may appear only in the rank 3 groups.
Finite index simplicial reflection subgroups of rank 3 simplicial groups are listed in~\cite{treug},
and we do not reproduce here the list for the groups with fundamental chambers having angles smaller than $\pi/6$.

For each of the non-visual maximal subgroups we also present an explicit embedding (Section~\ref{non-visual}).

Subgroup relations between the groups of the same rank are cited from~\cite{JKRT}.

\begin{theorem}
\label{classification of subgroups}
Let $H$ and $G$ be two simplicial groups, $3<\rank H\le \rank G$. $H$ is a subgroup of $G$ if and only if 
there exists a sequence of embeddings $H=K_0\subset K_1\subset \dots \subset K_l=G$, where all embeddings $K_i\subset K_{i+1}$
are maximal ones shown in Tables~\ref{3}--~\ref{6-5}.

\end{theorem}

\medskip 
\noindent
{\bf Notation.} We present  simplicial groups by the corresponding Coxeter diagrams.
If $H\subset G$ we draw $\Sigma(H)\to \Sigma(G)$, where   $\Sigma(G)$ is a Coxeter diagram of $G$ and  $\Sigma(H)$ is 
a Coxeter diagram of $H$.
To make the tables readable, we draw only maximal embeddings.

We use the following four types of arrows:

\epsfig{file=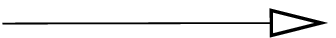,width=0.1\linewidth} \
for embeddings $H\subset G$, where 
$\rank H=\rank G$;

\epsfig{file=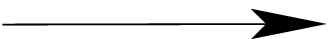,width=0.1\linewidth} \
for visual embeddings;

\epsfig{file=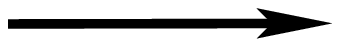,width=0.1\linewidth} \
for non-visual embeddings, where $\rank H=\rank G-1$;

\epsfig{file=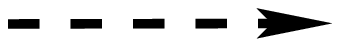,width=0.1\linewidth} \
for non-visual embeddings, where $\rank H<\rank G-1$.

\begin{remark} 
\label{class}
For each pair of simplicial groups $(H,G)$ we check if $H$ can be embedded into $G$.
In fact, if $H\subset G$ then $H$ may admit several different embeddings.
Our algorithm (Section~\ref{algorithm}) allows to list all of them,
however, the computation will be rather long. We restrict ourselves to the question
of existence of the embedding.

\end{remark}

\subsection{Subgroups of arithmetic over $\Q$ groups}
\label{ar}

\begin{table}[!b]
\begin{center}
\caption{Subgroups of simply-laced groups}
\label{3}
\psfrag{d}{ $\rank$  }
\psfrag{2}{$3$}
\psfrag{3}{$4$}
\psfrag{4}{$5$}
\psfrag{5}{$6$}
\psfrag{6}{$7$}
\psfrag{7}{$8$}
\psfrag{8}{$9$}
\psfrag{9}{$10$}
\epsfig{file=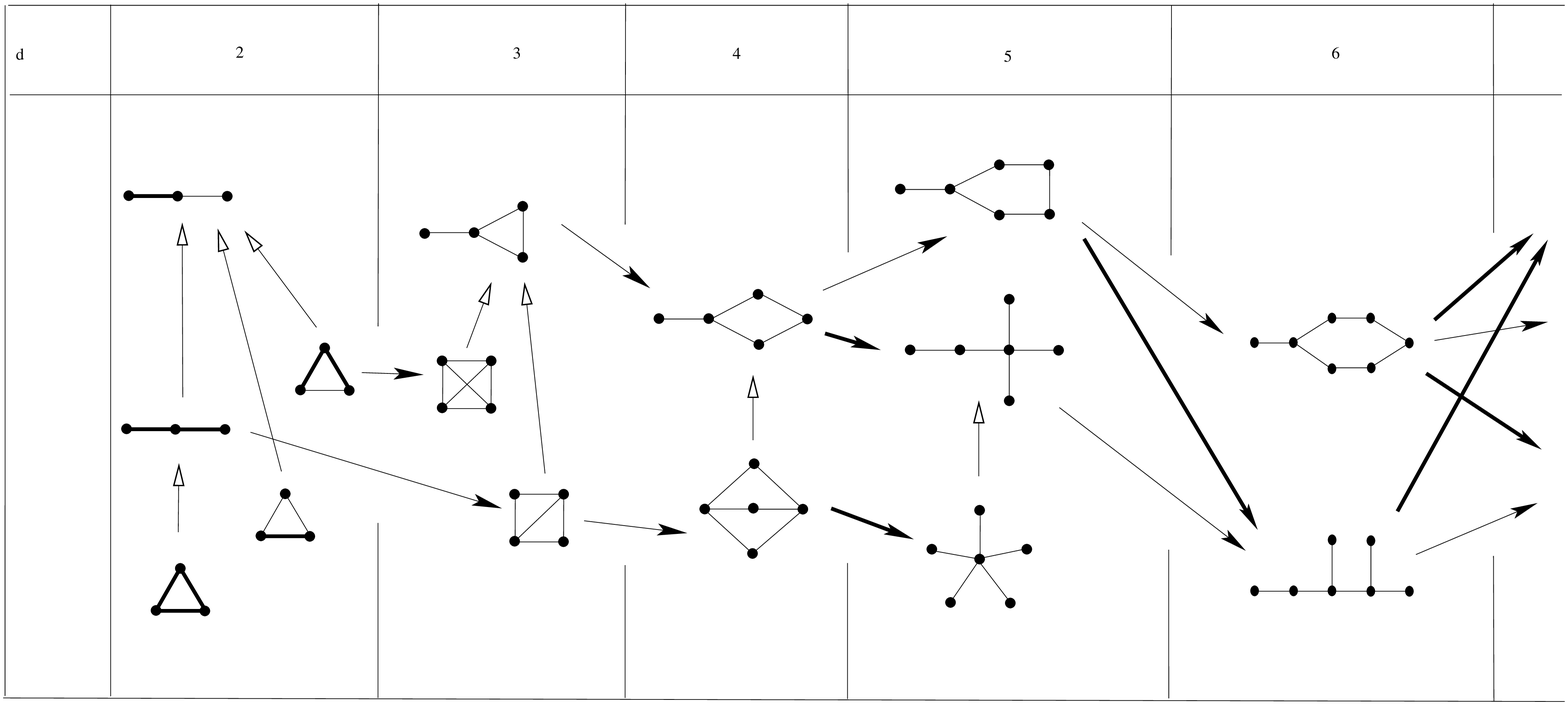,width=0.999\linewidth}

\vspace{30pt}
\psfrag{d}{$\rank$}
\psfrag{2}{$3$}
\psfrag{3}{$4$}
\psfrag{4}{$5$}
\psfrag{5}{$6$}
\psfrag{6}{$7$}
\psfrag{7}{$8$}
\psfrag{8}{$9$}
\psfrag{9}{$10$}
\epsfig{file=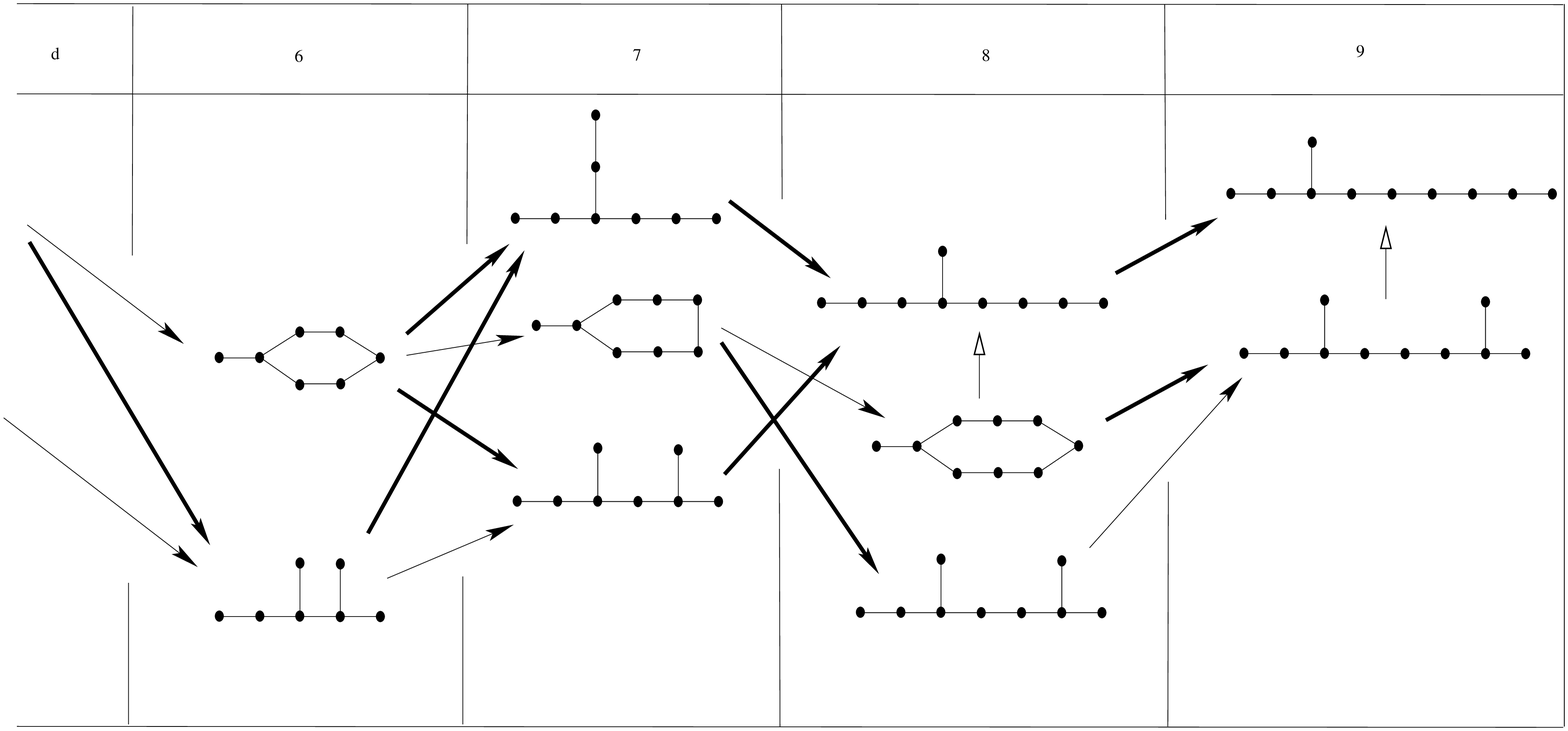,width=0.999\linewidth}
\end{center}
\end{table}

\begin{table}[!h]
\begin{center}
\caption{Maximal subgroups of rank 4 groups with $\pi/4$ (but without  $\pi/6$)}
\label{4_2in3}
\epsfig{file=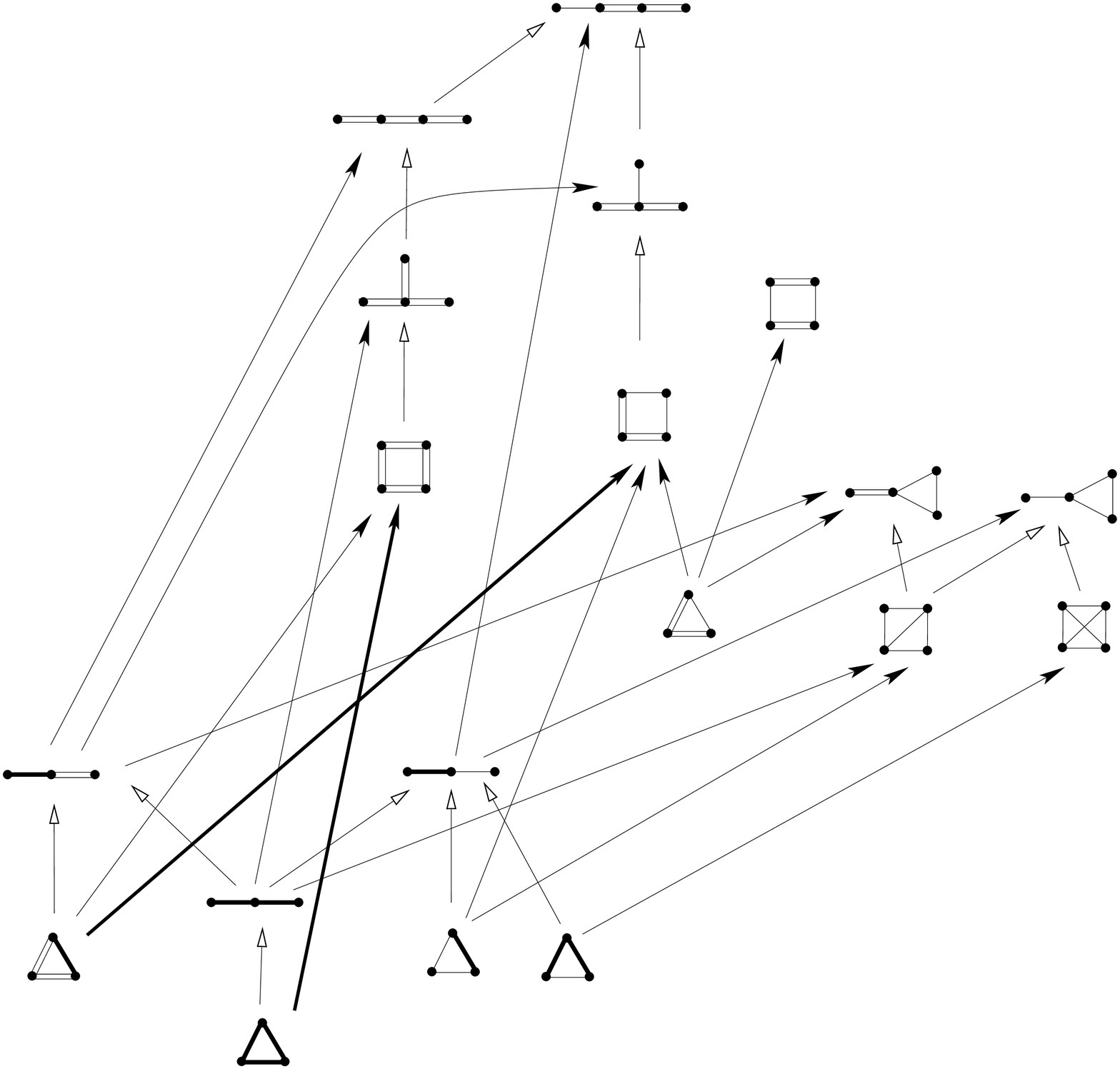,width=0.85\linewidth}
\end{center}
\end{table}

\begin{table}[!h]
\begin{center}
\caption{Maximal subgroups of  rank 5 groups with $\pi/4$ }
\label{4_3in4}
\epsfig{file=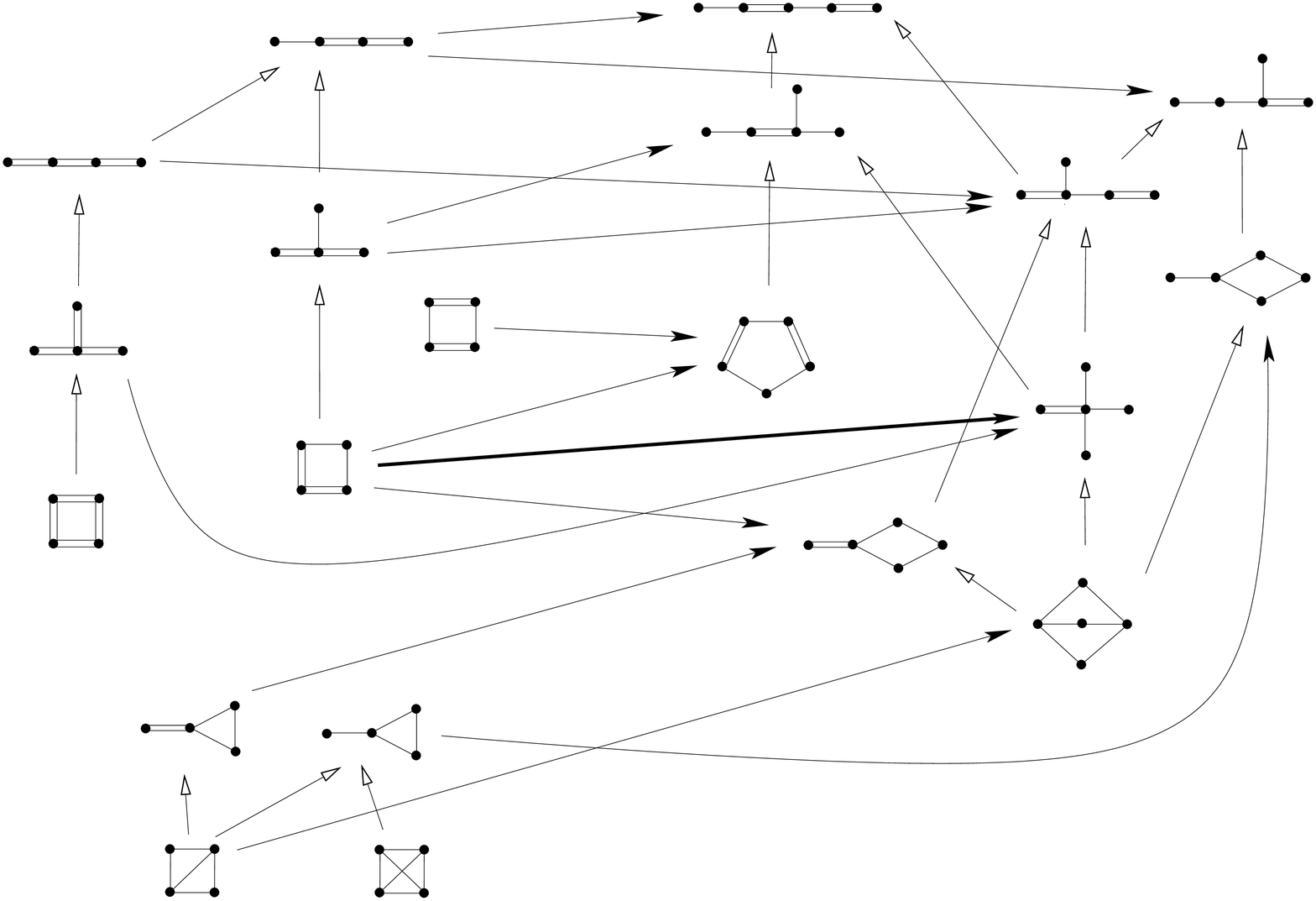,width=0.85\linewidth}
\end{center}
\end{table}

\begin{table}[!ht]
\caption{Maximal subgroups of rank 6 groups  with $\pi/4$ }
\label{4_4in5}
\epsfig{file=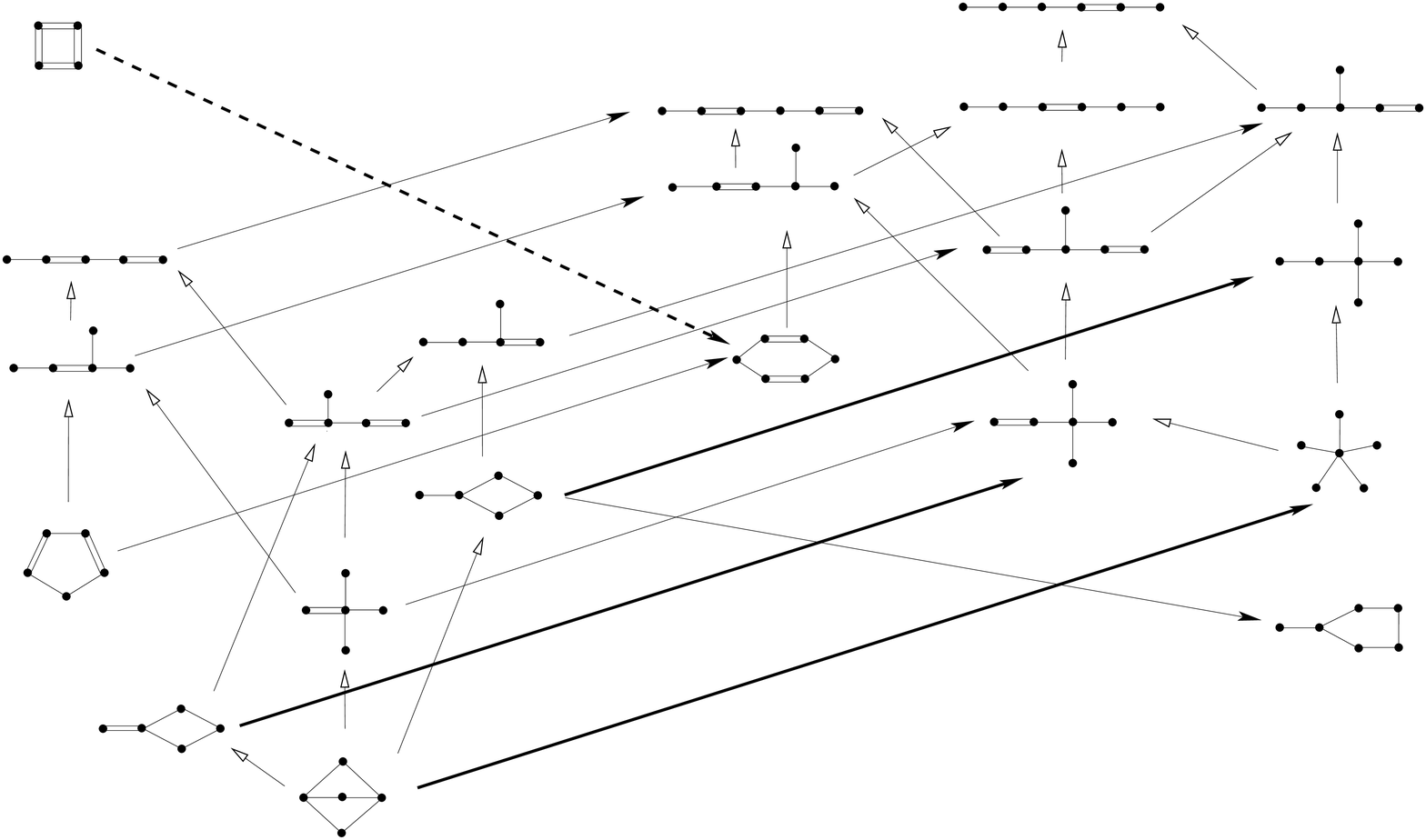,width=0.999\linewidth}
\end{table}

\begin{table}[!hb]
\caption{Maximal subgroups of rank 7 groups  with $\pi/4$}
\label{4_5in6}
\epsfig{file=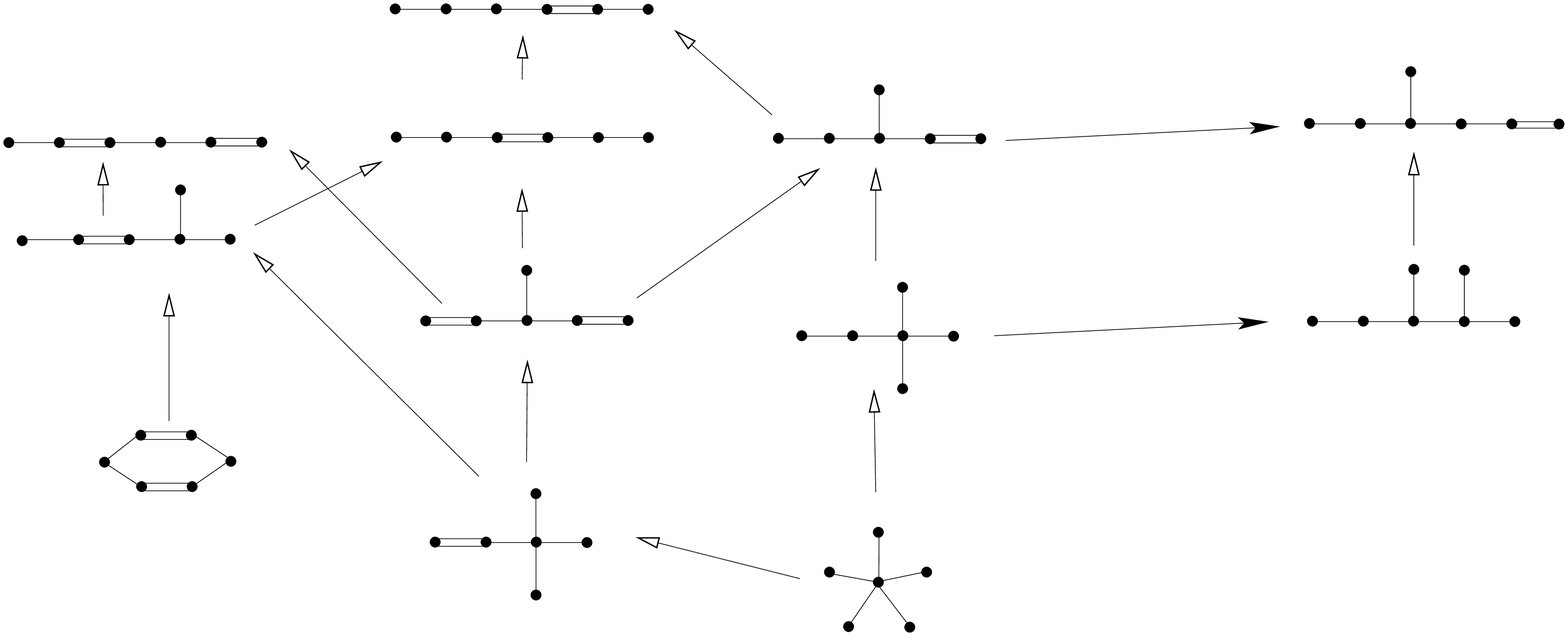,width=0.99\linewidth}
\end{table}

\begin{table}[!h]
\begin{center}
\caption{Maximal subgroups of groups of rank 8--10 with $\pi/4$ }
\label{4_6-9}
\psfrag{d}{$\rank$}
\psfrag{2}{$3$}
\psfrag{3}{$4$}
\psfrag{4}{$5$}
\psfrag{5}{$6$}
\psfrag{6}{$7$}
\psfrag{7}{$8$}
\psfrag{8}{$9$}
\psfrag{9}{$10$}
\epsfig{file=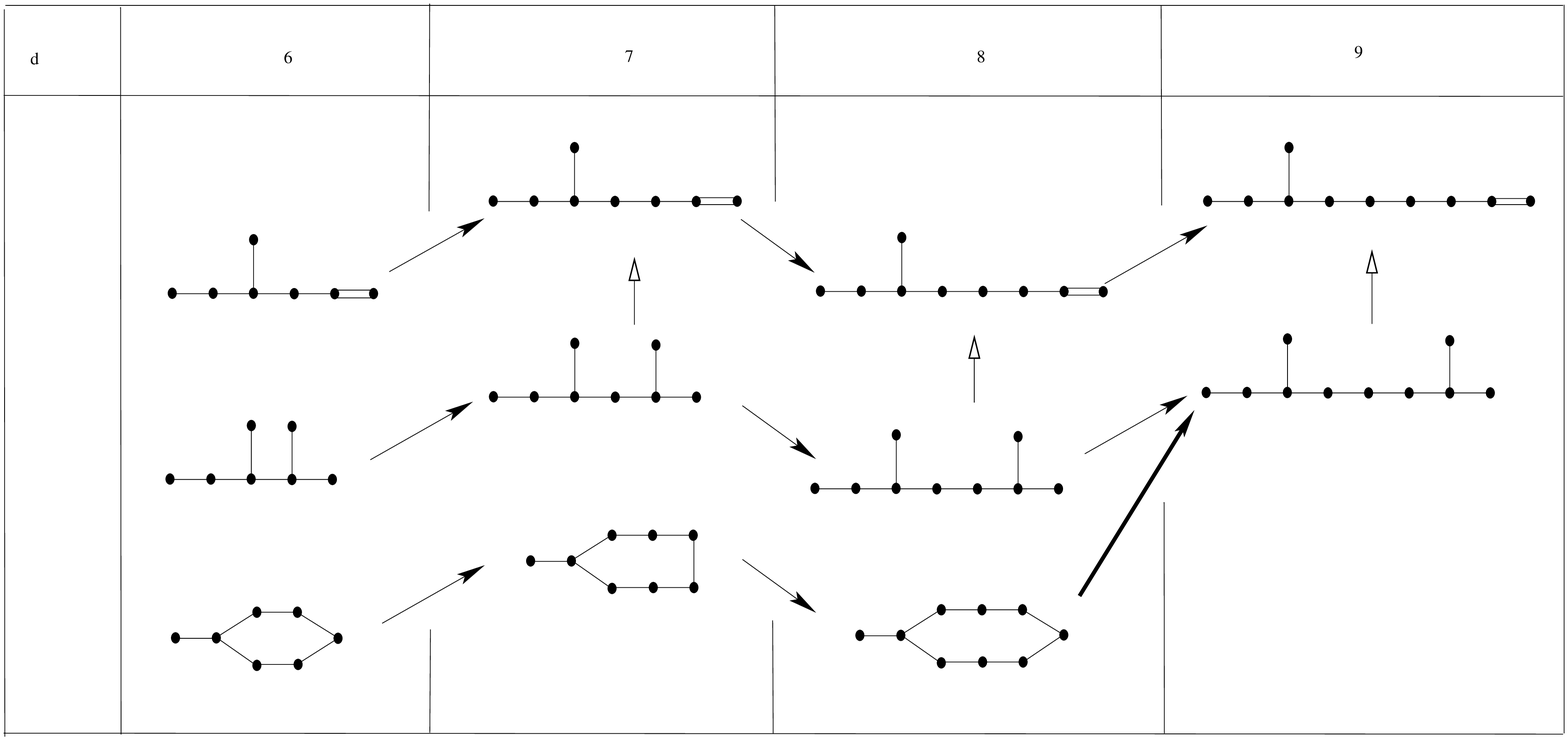,width=0.899\linewidth}
\end{center}
\end{table}


\begin{table}[!h]
\begin{center}
\caption{Maximal subgroups of groups with $\pi/6$.}
\label{6}
\epsfig{file=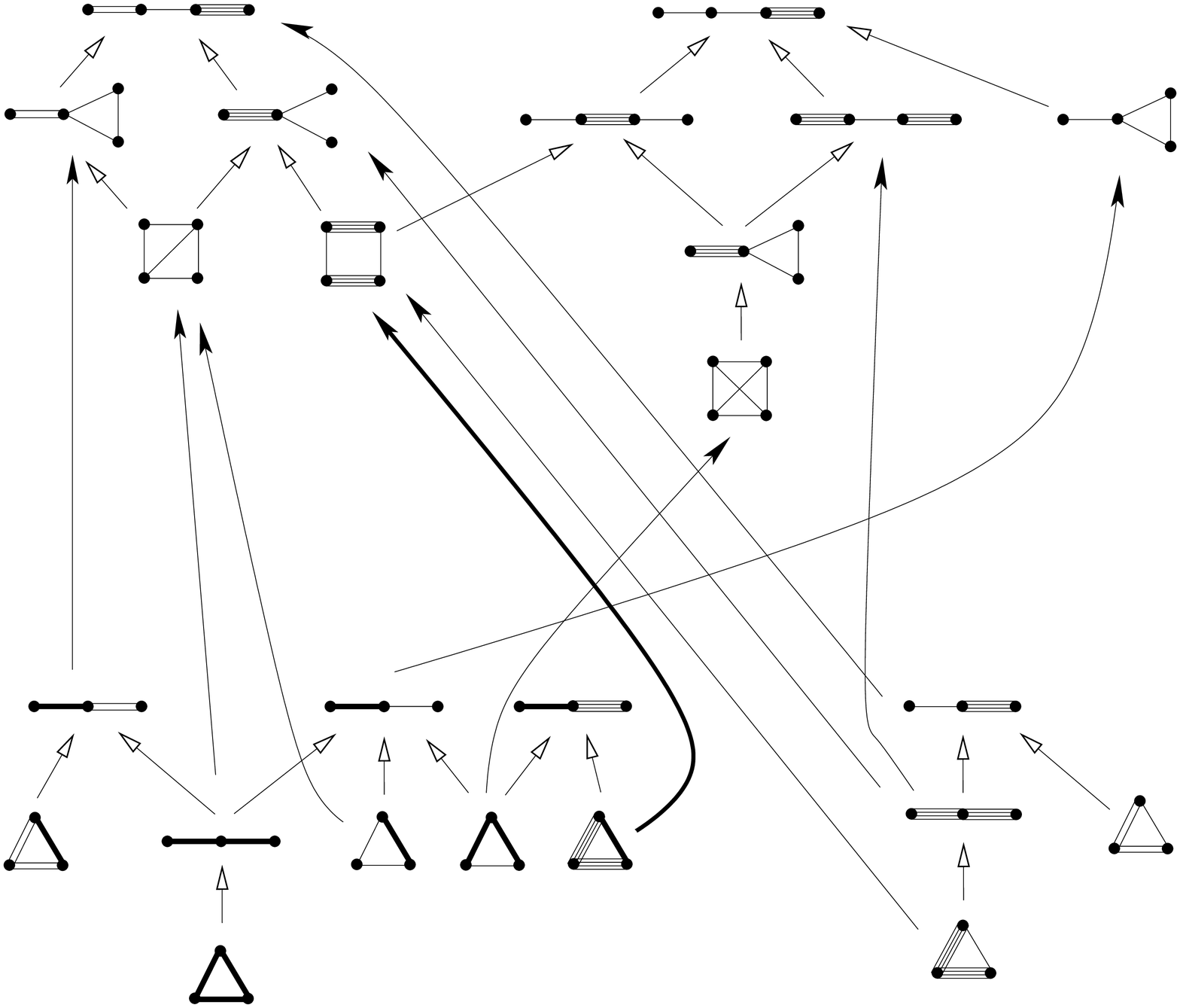,width=0.799\linewidth}
\end{center}
\end{table}

\clearpage

\subsection{Subgroups of non-arithmetic over $\Q$ groups}
\label{non-ar}


\begin{table}[!b]
\begin{center}
\caption{Subgroups of non-arithmetic over $\Q$ groups with $\pi/4$  (and without $\pi/5$
and $\pi/6$)}
\label{nonar}
\epsfig{file=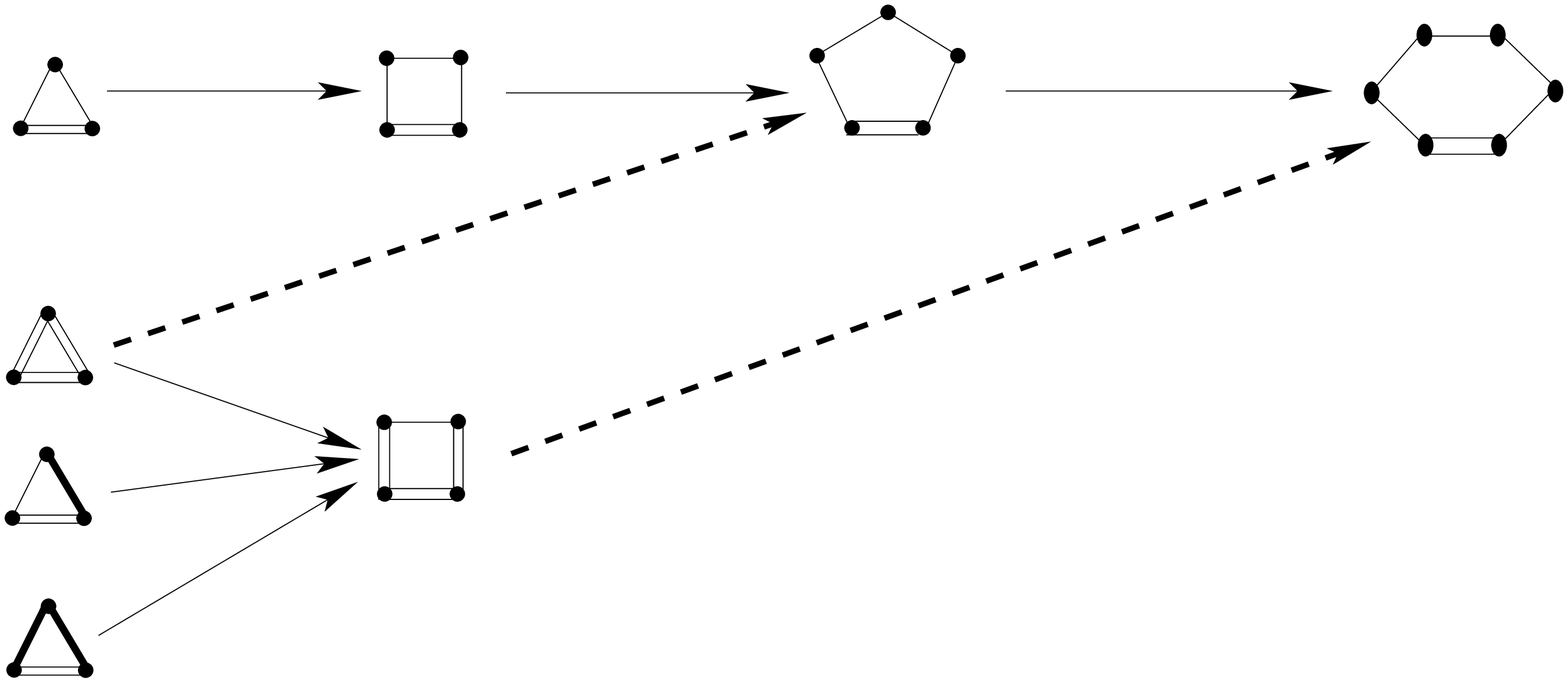,width=0.699\linewidth}
\end{center}
\end{table}

\begin{table}[!b]
\begin{center}
\caption{Subgroups of non-arithmetic over $\Q$ groups with $\pi/5$  (and without $\pi/6$)}
\label{5}
\epsfig{file=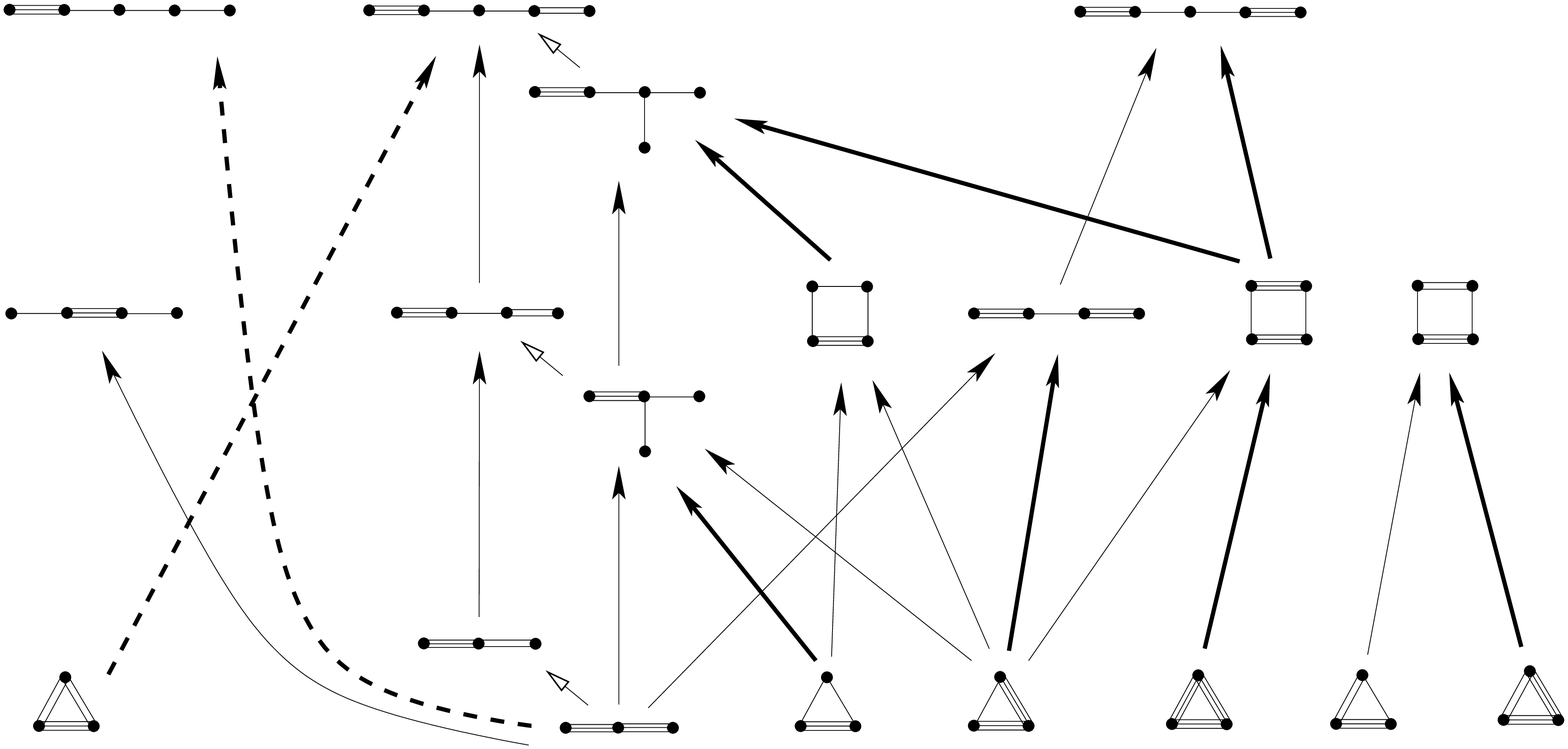,width=0.95\linewidth}
\end{center}
\end{table}

\begin{table}[!b]
\begin{center}
\caption{Subgroups of non-arithmetic over $\Q$ groups with $\pi/6$.}
\label{6-5}
\epsfig{file=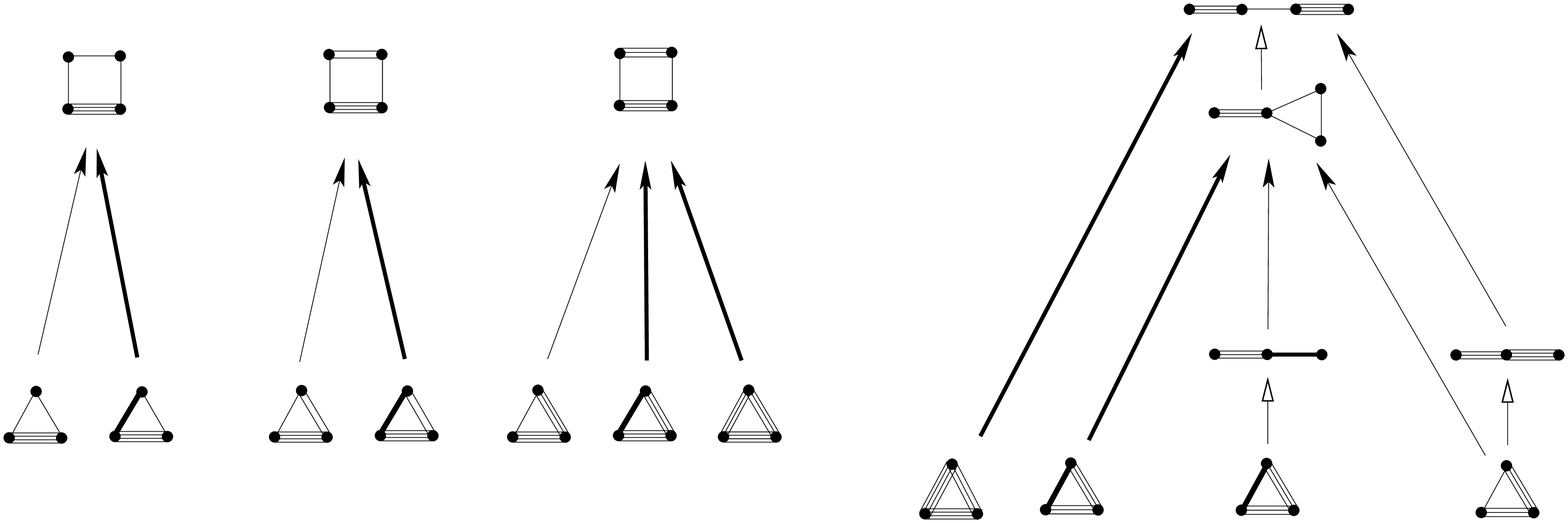,width=0.95\linewidth}
\end{center}
\end{table}


\clearpage
\pagebreak

\subsection{Non-visual subgroups} 
\label{non-visual}

In this section we present an explicit embedding for each  non-visual subgroup.


\begin{table}[!h]
\begin{center}
\caption{Non-visual subgroups}
\label{nonvis1}

\begin{tabular}{|c|c|c|c|}
\hline
& $H$ & $G$ & $u_*$ \\
\hline
1& 
\psfrag{u}{\tiny $u_*$}
\psfrag{1}{\tiny $1$}
\psfrag{2}{\tiny $2$}
\psfrag{3}{\tiny $3$}
\psfrag{4}{\tiny $4$}
\psfrag{5}{\tiny $5$}
\psfrag{6}{\tiny $6$}
\psfrag{7}{\tiny $7$}
\psfrag{8}{\tiny $8$}
\psfrag{9}{\tiny $9$}
\begin{tabular}{c}\epsfig{file=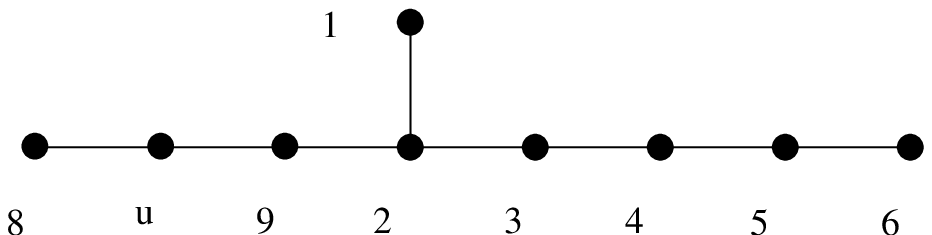,width=0.28\linewidth}\end{tabular}
& 
\psfrag{0}{\tiny $0$}
\psfrag{1}{\tiny $1$}
\psfrag{2}{\tiny $2$}
\psfrag{3}{\tiny $3$}
\psfrag{4}{\tiny $4$}
\psfrag{5}{\tiny $5$}
\psfrag{6}{\tiny $6$}
\psfrag{7}{\tiny $7$}
\psfrag{8}{\tiny $8$}
\psfrag{9}{\tiny $9$}
\begin{tabular}{c}\epsfig{file=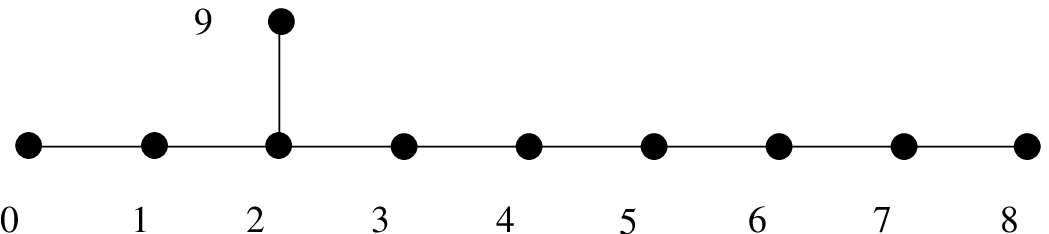,width=0.28\linewidth}\end{tabular}
& $r_7r_6r_5r_4r_3r_2r_1(v_0)$ \\

\hline
2& 
\psfrag{u}{\tiny $u_*$}
\psfrag{1}{\tiny $1$}
\psfrag{2}{\tiny $2$}
\psfrag{3}{\tiny $3$}
\psfrag{4}{\tiny $4$}
\psfrag{5}{\tiny $5$}
\psfrag{6}{\tiny $6$}
\psfrag{7}{\tiny $7$}
\psfrag{8}{\tiny $8$}
\psfrag{9}{\tiny $9$}
\begin{tabular}{c}\epsfig{file=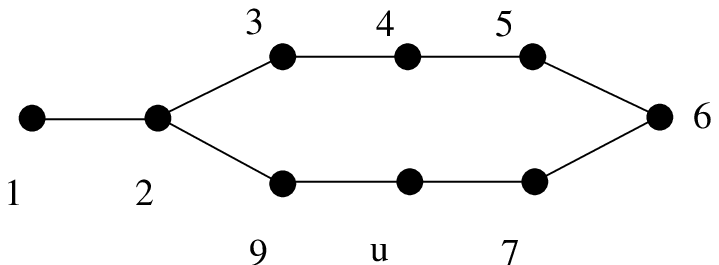,width=0.28\linewidth}\end{tabular}
& 
\psfrag{0}{\tiny $0$}
\psfrag{1}{\tiny $1$}
\psfrag{2}{\tiny $2$}
\psfrag{3}{\tiny $3$}
\psfrag{4}{\tiny $4$}
\psfrag{5}{\tiny $5$}
\psfrag{6}{\tiny $6$}
\psfrag{7}{\tiny $7$}
\psfrag{8}{\tiny $8$}
\psfrag{9}{\tiny $9$}
\begin{tabular}{c}\epsfig{file=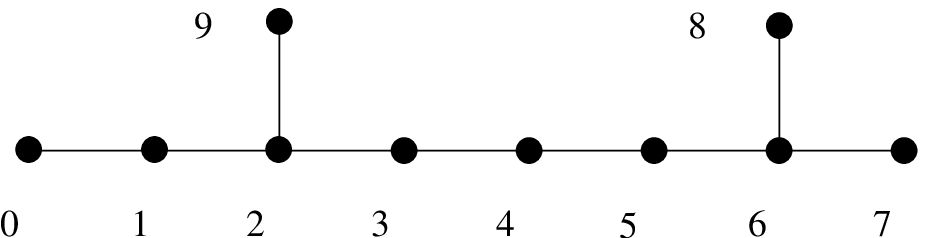,width=0.28\linewidth}\end{tabular}
& $r_0r_1r_2r_3r_4r_5r_6(v_8)$ \\

\hline
3& 
\psfrag{u}{\tiny $u_*$}
\psfrag{0}{\tiny $1$}
\psfrag{1}{\tiny $2$}
\psfrag{2}{\tiny $3$}
\psfrag{3}{\tiny $4$}
\psfrag{4}{\tiny $5$}
\psfrag{5}{\tiny $6$}
\psfrag{9}{\tiny $8$}
\begin{tabular}{c}\epsfig{file=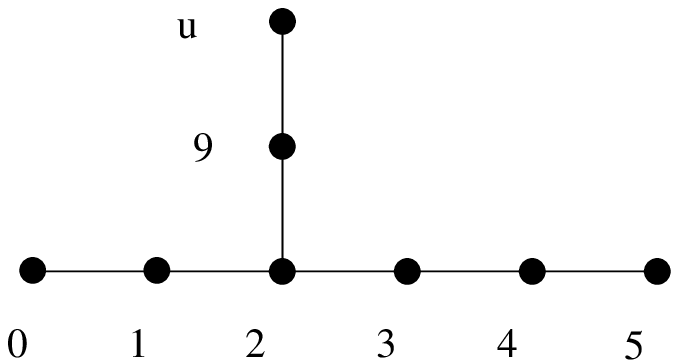,width=0.28\linewidth}\end{tabular}
& 
\psfrag{0}{\tiny $0$}
\psfrag{1}{\tiny $1$}
\psfrag{2}{\tiny $2$}
\psfrag{3}{\tiny $3$}
\psfrag{4}{\tiny $4$}
\psfrag{5}{\tiny $5$}
\psfrag{6}{\tiny $6$}
\psfrag{7}{\tiny $7$}
\psfrag{8}{\tiny $8$}
\psfrag{u}{\tiny $u$}
\begin{tabular}{c}\epsfig{file=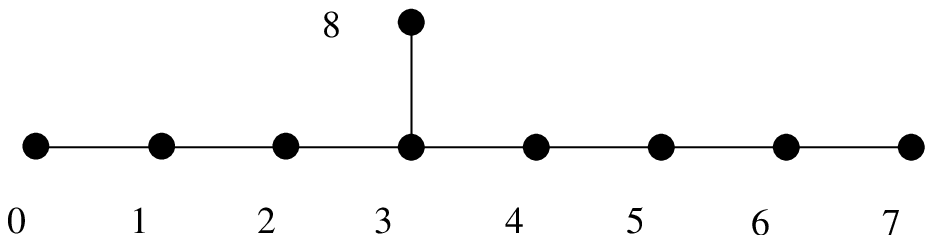,width=0.28\linewidth}\end{tabular}
& $r_0r_1r_2r_3r_4r_5r_6(v_7)$\\ 

\hline
4& 
\psfrag{0}{\tiny $0$}
\psfrag{1}{\tiny $1$}
\psfrag{2}{\tiny $2$}
\psfrag{3}{\tiny $3$}
\psfrag{4}{\tiny $4$}
\psfrag{5}{\tiny $5$}
\psfrag{6}{\tiny $6$}
\psfrag{7}{\tiny $7$}
\psfrag{8}{\tiny $8$}
\psfrag{u}{\tiny $u_*$}
\begin{tabular}{c}\epsfig{file=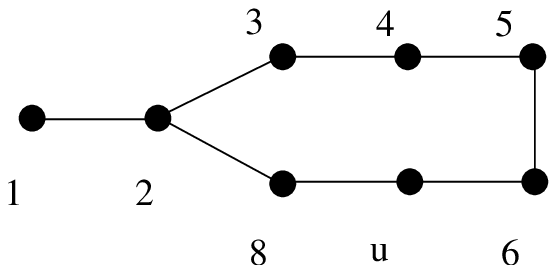,width=0.28\linewidth}\end{tabular}
& 
\psfrag{0}{\tiny $0$}
\psfrag{1}{\tiny $1$}
\psfrag{2}{\tiny $2$}
\psfrag{3}{\tiny $3$}
\psfrag{4}{\tiny $4$}
\psfrag{5}{\tiny $5$}
\psfrag{6}{\tiny $6$}
\psfrag{7}{\tiny $7$}
\psfrag{8}{\tiny $8$}
\psfrag{9}{\tiny $9$}
\begin{tabular}{c}\epsfig{file=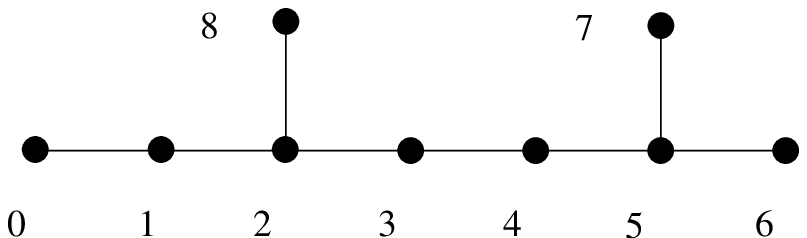,width=0.28\linewidth}\end{tabular}
& $r_0r_1r_2r_3r_4r_5(v_7)$ \\

\hline
5& 
\psfrag{0}{\tiny $7$}
\psfrag{1}{\tiny $6$}
\psfrag{2}{\tiny $5$}
\psfrag{3}{\tiny $4$}
\psfrag{4}{\tiny $3$}
\psfrag{5}{\tiny $2$}
\psfrag{6}{\tiny $8$}
\psfrag{7}{\tiny $u_*$}
\begin{tabular}{c}\epsfig{file=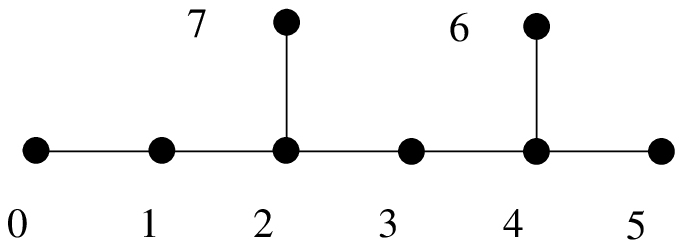,width=0.28\linewidth}\end{tabular}
& 
\psfrag{0}{\tiny $0$}
\psfrag{1}{\tiny $1$}
\psfrag{2}{\tiny $2$}
\psfrag{3}{\tiny $3$}
\psfrag{4}{\tiny $4$}
\psfrag{5}{\tiny $5$}
\psfrag{6}{\tiny $6$}
\psfrag{7}{\tiny $7$}
\psfrag{8}{\tiny $8$}
\begin{tabular}{c}\epsfig{file=./pic/nv_f3.eps,width=0.28\linewidth}\end{tabular}
& $r_1r_2r_3r_8r_4r_3r_2r_1(v_0)$ 
\\
\hline
6& 
\psfrag{0}{\tiny $0$}
\psfrag{1}{\tiny $1$}
\psfrag{2}{\tiny $2$}
\psfrag{3}{\tiny $3$}
\psfrag{4}{\tiny $4$}
\psfrag{5}{\tiny $5$}
\psfrag{6}{\tiny $6$}
\psfrag{u}{\tiny $u_*$}
\begin{tabular}{c}\epsfig{file=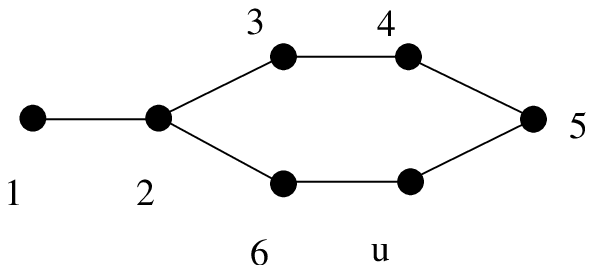,width=0.28\linewidth}\end{tabular}
& 
\psfrag{0}{\tiny $0$}
\psfrag{1}{\tiny $1$}
\psfrag{2}{\tiny $2$}
\psfrag{3}{\tiny $3$}
\psfrag{4}{\tiny $4$}
\psfrag{5}{\tiny $5$}
\psfrag{9}{\tiny $6$}
\psfrag{u}{\tiny $7$}
\begin{tabular}{c}\epsfig{file=./pic/nv_p3.eps,width=0.28\linewidth}\end{tabular}
& $r_0r_1r_2r_3r_4(v_5)$ \\

\hline
7& 
\psfrag{1}{\tiny $0$}
\psfrag{2}{\tiny $1$}
\psfrag{3}{\tiny $2$}
\psfrag{4}{\tiny $3$}
\psfrag{5}{\tiny $4$}
\psfrag{6}{\tiny $u_*$}
\psfrag{u}{\tiny $5$}
\begin{tabular}{c}\epsfig{file=./pic/nv_p6.eps,width=0.28\linewidth}\end{tabular}
& 
\psfrag{0}{\tiny $0$}
\psfrag{1}{\tiny $1$}
\psfrag{2}{\tiny $2$}
\psfrag{3}{\tiny $3$}
\psfrag{4}{\tiny $4$}
\psfrag{5}{\tiny $5$}
\psfrag{6}{\tiny $6$}
\psfrag{7}{\tiny $7$}
\begin{tabular}{c}\epsfig{file=./pic/nv_f7.eps,width=0.28\linewidth}\end{tabular}
& $r_6r_4r_3r_2(v_7)$ \\

\hline
8& 
\psfrag{0}{\tiny $5$}
\psfrag{1}{\tiny $4$}
\psfrag{2}{\tiny $3$}
\psfrag{3}{\tiny $2$}
\psfrag{4}{\tiny $1$}
\psfrag{5}{\tiny $6$}
\psfrag{6}{\tiny $u_*$}
\begin{tabular}{c}\epsfig{file=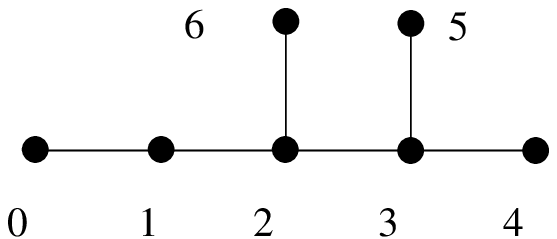,width=0.28\linewidth}\end{tabular}
& 
\psfrag{0}{\tiny $0$}
\psfrag{1}{\tiny $1$}
\psfrag{2}{\tiny $2$}
\psfrag{3}{\tiny $3$}
\psfrag{4}{\tiny $4$}
\psfrag{5}{\tiny $5$}
\psfrag{9}{\tiny $6$}
\psfrag{u}{\tiny $7$}
\begin{tabular}{c}\epsfig{file=./pic/nv_p3.eps,width=0.28\linewidth}\end{tabular}
& $r_0r_1r_2r_6(v_7)$ \\

\hline
9& 
\psfrag{1}{\tiny $1$}
\psfrag{2}{\tiny $2$}
\psfrag{3}{\tiny $3$}
\psfrag{4}{\tiny $4$}
\psfrag{u}{\tiny $u_*$}
\psfrag{6}{\tiny $6$}
\begin{tabular}{c}\epsfig{file=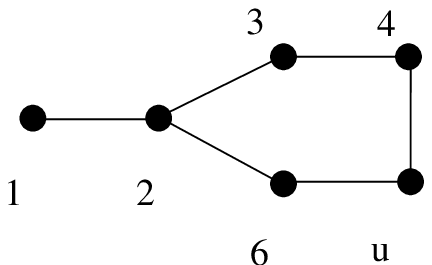,width=0.28\linewidth}\end{tabular}
& 
\psfrag{0}{\tiny $0$}
\psfrag{1}{\tiny $1$}
\psfrag{2}{\tiny $2$}
\psfrag{3}{\tiny $3$}
\psfrag{4}{\tiny $4$}
\psfrag{5}{\tiny $5$}
\psfrag{6}{\tiny $6$}
\begin{tabular}{c}\epsfig{file=./pic/nv_f9.eps,width=0.28\linewidth}\end{tabular}
& $r_0r_1r_2r_3(v_5)$ \\

\hline
10& 
\psfrag{0}{\tiny $0$}
\psfrag{u}{\tiny $u_*$}
\psfrag{2}{\tiny $2$}
\psfrag{3}{\tiny $3$}
\psfrag{4}{\tiny $4$}
\psfrag{5}{\tiny $5$}
\begin{tabular}{c}\epsfig{file=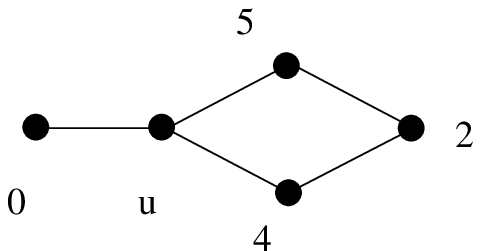,width=0.28\linewidth}\end{tabular}
& 
\psfrag{0}{\tiny $0$}
\psfrag{1}{\tiny $1$}
\psfrag{2}{\tiny $2$}
\psfrag{3}{\tiny $3$}
\psfrag{4}{\tiny $4$}
\psfrag{5}{\tiny $5$}
\begin{tabular}{c}\epsfig{file=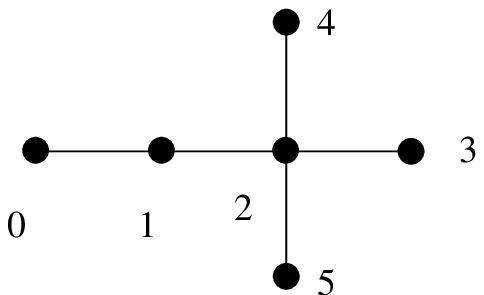,width=0.28\linewidth}\end{tabular}
& $r_1r_2(v_3)$ \\

\hline
11& 
\psfrag{0}{\tiny $0$}
\psfrag{u}{\tiny $u_*$}
\psfrag{2}{\tiny $2$}
\psfrag{3}{\tiny $3$}
\psfrag{4}{\tiny $4$}
\psfrag{5}{\tiny $5$}
\begin{tabular}{c}\epsfig{file=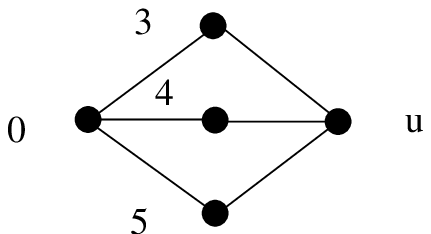,width=0.28\linewidth}\end{tabular}
& 
\psfrag{0}{\tiny $0$}
\psfrag{1}{\tiny $1$}
\psfrag{2}{\tiny $2$}
\psfrag{3}{\tiny $3$}
\psfrag{4}{\tiny $4$}
\psfrag{5}{\tiny $5$}
\begin{tabular}{c}\epsfig{file=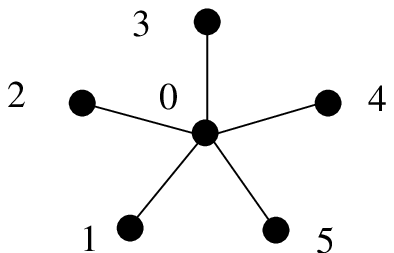,width=0.28\linewidth}\end{tabular}
& $r_1r_0(v_2)$ \\

\hline
\end{tabular}
\end{center}
\end{table}

\pagebreak
\clearpage
\addtocounter{table}{-1}

\begin{table}[!h]
\begin{center}
\caption{ Cont.}
\label{nonvis2}
\begin{tabular}{|c|c|c|c|}
\hline
& $H$ & $G$ & $u_*$ \\
\hline

\hline
12& 
\psfrag{0}{\tiny $0$}
\psfrag{1}{\tiny $1$}
\psfrag{2}{\tiny $2$}
\psfrag{u}{\tiny $u_*$}
\psfrag{4}{\tiny $4$}
\psfrag{5}{\tiny $5$}
\begin{tabular}{c}\epsfig{file=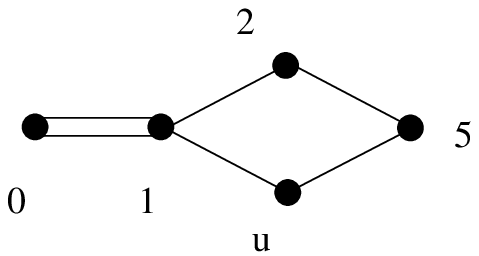,width=0.2\linewidth}\end{tabular}
& 
\psfrag{0}{\tiny $0$}
\psfrag{1}{\tiny $1$}
\psfrag{2}{\tiny $2$}
\psfrag{3}{\tiny $3$}
\psfrag{4}{\tiny $4$}
\psfrag{5}{\tiny $5$}
\begin{tabular}{c}\raisebox{-0pt}[47pt][0pt]{\epsfig{file=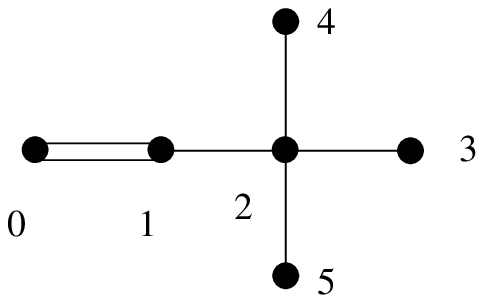,width=0.2\linewidth}}\end{tabular}
& $r_4r_2(v_3)$ \\

\hline
13& 
\psfrag{0}{\tiny $0$}
\psfrag{1}{\tiny $1$}
\psfrag{2}{\tiny $2$}
\psfrag{u}{\tiny $u_*$}
\psfrag{4}{\tiny $4$}
\begin{tabular}{c}\epsfig{file=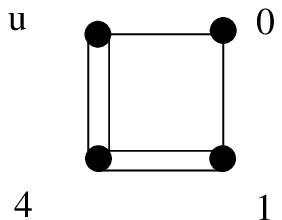,width=0.2\linewidth}\end{tabular}
& 
\psfrag{0}{\tiny $0$}
\psfrag{1}{\tiny $1$}
\psfrag{2}{\tiny $2$}
\psfrag{3}{\tiny $3$}
\psfrag{4}{\tiny $4$}
\begin{tabular}{c}\raisebox{-0pt}[47pt][0pt]{\epsfig{file=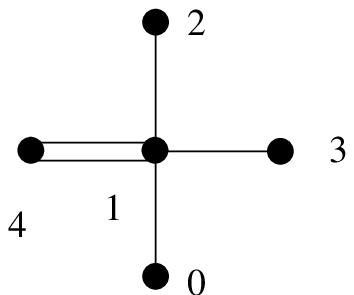,width=0.2\linewidth}}\end{tabular}
& $r_2r_1(v_0)$ \\

\hline
14& 
\psfrag{0}{\tiny $0$}
\psfrag{1}{\tiny $v_*$}
\psfrag{2}{\tiny $4$}
\psfrag{3}{\tiny $u_*$}
\begin{tabular}{c}\epsfig{file=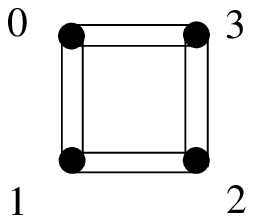,width=0.2\linewidth}\end{tabular}
& 
\psfrag{0}{\tiny $0$}
\psfrag{1}{\tiny $1$}
\psfrag{2}{\tiny $2$}
\psfrag{3}{\tiny $3$}
\psfrag{4}{\tiny $4$}
\psfrag{5}{\tiny $5$}
\begin{tabular}{c}\raisebox{-0pt}[40pt][0pt]{\epsfig{file=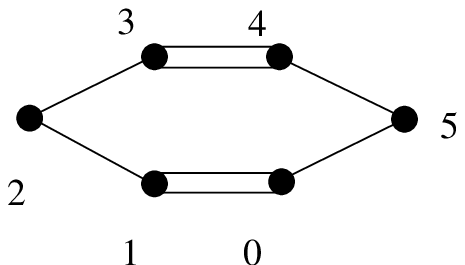,width=0.2\linewidth}}\end{tabular}
&
\begin{tabular}{l}
$v_*=r_1r_2(v_3)$\\
$u_*=r_3r_1r_4r_5r_0r_3r_2(v_1)$
\end{tabular}\\

\hline
15& 
\psfrag{0}{\tiny $v_*$}
\psfrag{1}{\tiny $1$}
\psfrag{u}{\tiny $u_*$}
\begin{tabular}{c}\epsfig{file=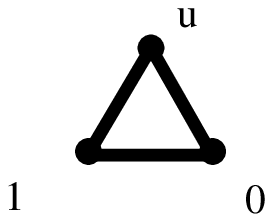,width=0.2\linewidth}\end{tabular}
& 
\psfrag{0}{\tiny $0$}
\psfrag{1}{\tiny $1$}
\psfrag{2}{\tiny $2$}
\psfrag{3}{\tiny $3$}
\begin{tabular}{c}\epsfig{file=./pic/nv_f15.eps,width=0.2\linewidth}\end{tabular}
&
\begin{tabular}{l}
$v_*=r_2(v_3)$\\
$u_*=r_0r_2(v_1)$
\end{tabular}\\
 
\hline
16& 
\psfrag{0}{\tiny $1$}
\psfrag{1}{\tiny $2$}
\psfrag{u}{\tiny $u_*$}
\begin{tabular}{c}\epsfig{file=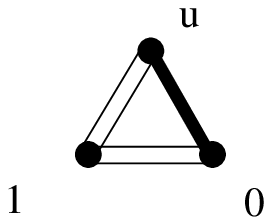,width=0.2\linewidth}\end{tabular}
& 
\psfrag{0}{\tiny $0$}
\psfrag{1}{\tiny $1$}
\psfrag{2}{\tiny $2$}
\psfrag{3}{\tiny $3$}
\begin{tabular}{c}\epsfig{file=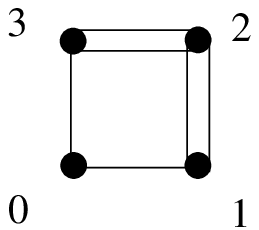,width=0.2\linewidth}\end{tabular}
& $r_0r_1r_3(v_2)$ \\ 

\hline
17& 
\psfrag{0}{\tiny $2$}
\psfrag{1}{\tiny $u_*$}
\psfrag{u}{\tiny $3$}
\psfrag{k}{\tiny }
\begin{tabular}{c}\epsfig{file=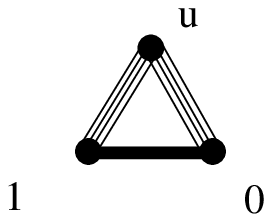,width=0.2\linewidth}\end{tabular}
& 
\psfrag{0}{\tiny $0$}
\psfrag{1}{\tiny $1$}
\psfrag{2}{\tiny $2$}
\psfrag{3}{\tiny $3$}
\psfrag{k}{\tiny }
\begin{tabular}{c}\epsfig{file=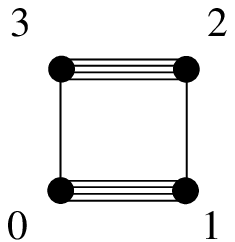,width=0.2\linewidth}\end{tabular}
& $r_0r_3r_2(v_1)$ \\

\hline 
&&&\\
\hline
18& 
\psfrag{0}{\tiny $0$}
\psfrag{1}{\tiny $1$}
\psfrag{u}{\tiny $u_*$}
\begin{tabular}{c}\epsfig{file=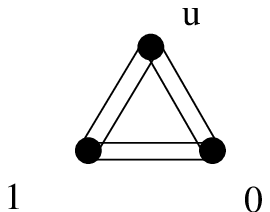,width=0.2\linewidth}\end{tabular}
& 
\psfrag{0}{\tiny $0$}
\psfrag{1}{\tiny $1$}
\psfrag{2}{\tiny $2$}
\psfrag{3}{\tiny $3$}
\psfrag{4}{\tiny $4$}
\psfrag{5}{\tiny $5$}
\begin{tabular}{c}\epsfig{file=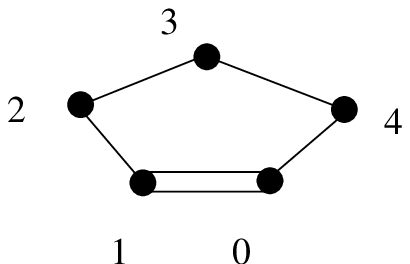,width=0.2\linewidth}\end{tabular}
& $r_2r_3r_4r_1r_2r_0(v_1)$ \\
 
\hline
19& 
\psfrag{0}{\tiny $1$}
\psfrag{1}{\tiny $2$}
\psfrag{2}{\tiny $3$}
\psfrag{3}{\tiny $u_*$}
\begin{tabular}{c}\epsfig{file=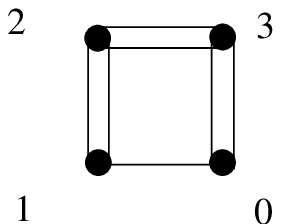,width=0.2\linewidth}\end{tabular}
& 
\psfrag{0}{\tiny $0$}
\psfrag{1}{\tiny $1$}
\psfrag{2}{\tiny $2$}
\psfrag{3}{\tiny $3$}
\psfrag{4}{\tiny $4$}
\psfrag{5}{\tiny $5$}
\begin{tabular}{c}\epsfig{file=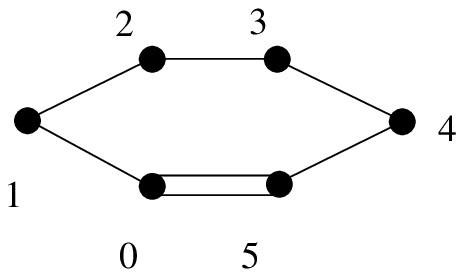,width=0.2\linewidth}\end{tabular}
& $r_0r_5r_4r_1r_0r_2r_3r_2(v_1)$ \\

\hline
20& 
\psfrag{0}{\tiny $2$}
\psfrag{1}{\tiny $1$}
\psfrag{u}{\tiny $u_*$}
\begin{tabular}{c}\epsfig{file=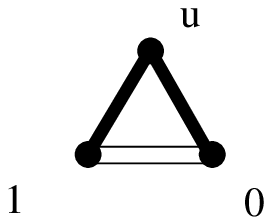,width=0.2\linewidth}\end{tabular}
& 
\psfrag{0}{\tiny $0$}
\psfrag{1}{\tiny $1$}
\psfrag{2}{\tiny $2$}
\psfrag{3}{\tiny $3$}
\begin{tabular}{c}\epsfig{file=./pic/nv_f19.eps,width=0.2\linewidth}\end{tabular}
& $r_0r_1r_3(v_2)$ \\

\hline
\end{tabular}
\end{center}
\end{table}

\pagebreak
\clearpage
\addtocounter{table}{-1}

\begin{table}[htb!]
\begin{center}
\caption{ Cont.}
\label{nonvis3}
\begin{tabular}{|c|c|c|c|}
\hline
& $H$ & $G$ & $u_*$ \\
\hline

21& 
\psfrag{0}{\tiny $3$}
\psfrag{1}{\tiny $2$}
\psfrag{u}{\tiny $u_*$}
\psfrag{k}{\tiny $k$}
\begin{tabular}{c}\epsfig{file=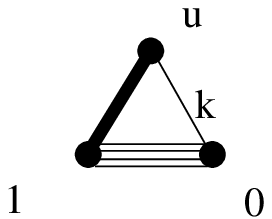,width=0.2\linewidth}\end{tabular}
& 
\psfrag{0}{\tiny $0$}
\psfrag{1}{\tiny $1$}
\psfrag{2}{\tiny $2$}
\psfrag{3}{\tiny $3$}
\psfrag{k}{\tiny $k=4,5,6$}
\begin{tabular}{c}\raisebox{-0pt}[35pt][0pt]{\epsfig{file=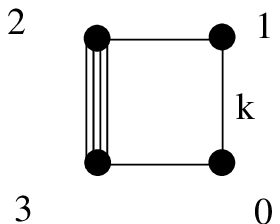,width=0.2\linewidth}}\end{tabular}
& $r_0r_3r_2(v_1)$ \\

\hline
22& 
\psfrag{0}{\tiny $3$}
\psfrag{1}{\tiny $2$}
\psfrag{u}{\tiny $u_*$}
\psfrag{k}{\tiny $k$}
\begin{tabular}{c}\epsfig{file=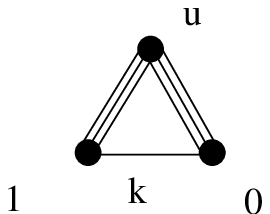,width=0.2\linewidth}\end{tabular}
& 
\psfrag{0}{\tiny $0$}
\psfrag{1}{\tiny $1$}
\psfrag{2}{\tiny $2$}
\psfrag{3}{\tiny $3$}
\psfrag{k}{\tiny $k=3,4,5,6$}
\begin{tabular}{c}\raisebox{-0pt}[35pt][0pt]{\epsfig{file=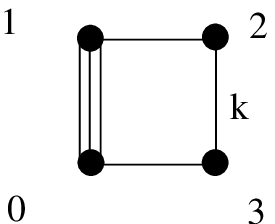,width=0.2\linewidth}}\end{tabular}
& $r_0r_1(v_0)$ \\

\hline
23& 
\psfrag{0}{\tiny $0$}
\psfrag{1}{\tiny $1$}
\psfrag{u}{\tiny $u_*$}
\begin{tabular}{c}\epsfig{file=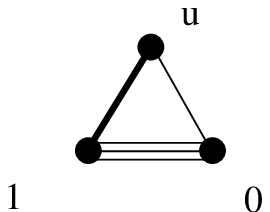,width=0.2\linewidth}\end{tabular}
& 
\psfrag{0}{\tiny $0$}
\psfrag{1}{\tiny $1$}
\psfrag{2}{\tiny $2$}
\psfrag{3}{\tiny $3$}
\begin{tabular}{c}\epsfig{file=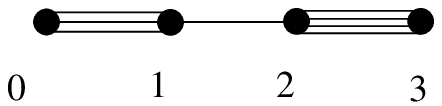,width=0.2\linewidth}\end{tabular}
& $r_2r_3r_1r_0r_1(v_2)$ \\

\hline
24& 
\psfrag{0}{\tiny $0$}
\psfrag{1}{\tiny $1$}
\psfrag{u}{\tiny $u_*$}
\begin{tabular}{c}\epsfig{file=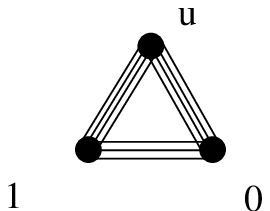,width=0.2\linewidth}\end{tabular}
& 
\psfrag{0}{\tiny $0$}
\psfrag{1}{\tiny $1$}
\psfrag{2}{\tiny $2$}
\psfrag{3}{\tiny $3$}
\begin{tabular}{c}\epsfig{file=./pic/nv_f22.eps,width=0.2\linewidth}\end{tabular}
&  $r_2r_3r_1r_0r_1r_2(v_3)$ \\

\hline

25& 
\psfrag{u}{\tiny $u_*$}
\psfrag{2}{\tiny $2$}
\psfrag{3}{\tiny $3$}
\psfrag{4}{\tiny $4$}
\begin{tabular}{c}\epsfig{file=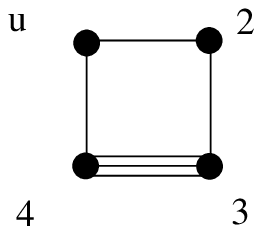,width=0.2\linewidth}\end{tabular}
& 
\psfrag{0}{\tiny $0$}
\psfrag{1}{\tiny $1$}
\psfrag{2}{\tiny $2$}
\psfrag{3}{\tiny $3$}
\psfrag{4}{\tiny $4$}
\begin{tabular}{c}\epsfig{file=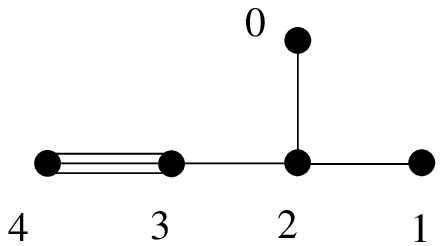,width=0.2\linewidth}\end{tabular}
&  $r_1r_2r_3r_4r_3r_2(v_0)$ \\

\hline
26& 
\psfrag{u}{\tiny $u_*$}
\psfrag{2}{\tiny $2$}
\psfrag{3}{\tiny $3$}
\psfrag{4}{\tiny $4$}
\begin{tabular}{c}\epsfig{file=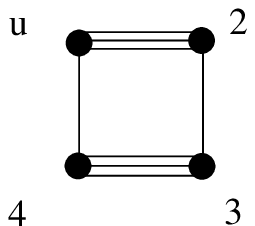,width=0.2\linewidth}\end{tabular}
& 
\psfrag{0}{\tiny $0$}
\psfrag{1}{\tiny $1$}
\psfrag{2}{\tiny $2$}
\psfrag{3}{\tiny $3$}
\psfrag{4}{\tiny $4$}
\begin{tabular}{c}\epsfig{file=./pic/nv_f24.eps,width=0.2\linewidth}\end{tabular}
&  $r_0r_1r_2r_3r_4r_3r_2(v_0)$ \\

\hline
27& 
\psfrag{u}{\tiny $u_*$}
\psfrag{2}{\tiny $2$}
\psfrag{3}{\tiny $3$}
\psfrag{4}{\tiny $4$}
\begin{tabular}{c}\epsfig{file=./pic/nv_p25.eps,width=0.2\linewidth}\end{tabular}
& 
\psfrag{0}{\tiny $0$}
\psfrag{1}{\tiny $1$}
\psfrag{2}{\tiny $2$}
\psfrag{3}{\tiny $3$}
\psfrag{4}{\tiny $4$}
\begin{tabular}{c}\epsfig{file=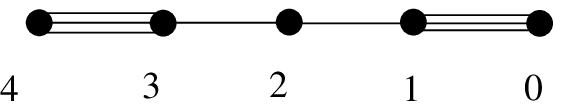,width=0.2\linewidth}\end{tabular}
&  $r_1r_2r_3r_4r_3r_2r_0(v_1)$ \\

\hline
28& 
\psfrag{3}{\tiny $3$}
\psfrag{4}{\tiny $u_*$}
\psfrag{u}{\tiny $1$}
\begin{tabular}{c}\epsfig{file=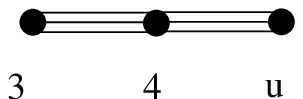,width=0.2\linewidth}\end{tabular}
& 
\psfrag{0}{\tiny $0$}
\psfrag{1}{\tiny $1$}
\psfrag{2}{\tiny $2$}
\psfrag{3}{\tiny $3$}
\psfrag{4}{\tiny $4$}
\begin{tabular}{c}\epsfig{file=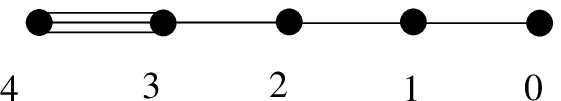,width=0.2\linewidth}\end{tabular}
&  $r_4r_2r_3r_4r_3r_2r_1(v_0)$ \\

\hline
29& 
\psfrag{u}{\tiny $0$}
\psfrag{1}{\tiny $1$}
\psfrag{0}{\tiny $u_*$}
\begin{tabular}{c}\epsfig{file=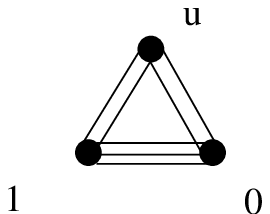,width=0.2\linewidth}\end{tabular}
& 
\psfrag{0}{\tiny $0$}
\psfrag{1}{\tiny $1$}
\psfrag{2}{\tiny $2$}
\psfrag{3}{\tiny $3$}
\psfrag{4}{\tiny $4$}
\begin{tabular}{c}\epsfig{file=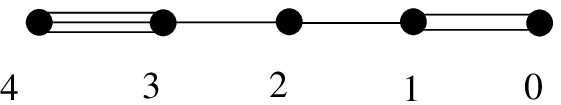,width=0.2\linewidth}\end{tabular}
&  $r_2r_3r_4r_3r_2r_4r_3r_4r_3r_2(v_1)$ \\

\hline
30& 
\psfrag{u}{\tiny $2$}
\psfrag{0}{\tiny $u_*$}
\psfrag{1}{\tiny $1$}
\begin{tabular}{c}\epsfig{file=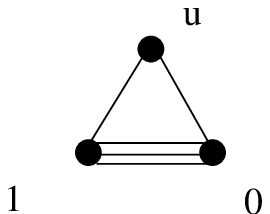,width=0.2\linewidth}\end{tabular}
& 
\psfrag{0}{\tiny $0$}
\psfrag{1}{\tiny $1$}
\psfrag{2}{\tiny $2$}
\psfrag{3}{\tiny $3$}
\begin{tabular}{c}\epsfig{file=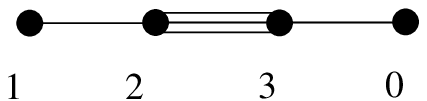,width=0.2\linewidth}\end{tabular}
&  $r_3r_2r_3(v_0)$ \\

\hline
31& 
\psfrag{u}{\tiny $3$}
\psfrag{3}{\tiny $0$}
\psfrag{4}{\tiny $u_*$}
\begin{tabular}{c}\epsfig{file=./pic/nv_p27.eps,width=0.2\linewidth}\end{tabular}
& 
\psfrag{0}{\tiny $3$}
\psfrag{1}{\tiny $0$}
\psfrag{2}{\tiny $1$}
\psfrag{3}{\tiny $2$}
\begin{tabular}{c}\epsfig{file=./pic/nv_f29.eps,width=0.2\linewidth}\end{tabular}
&  $r_1r_2(v_1)$ \\

\hline
32& 
\psfrag{3}{\tiny $2$}
\psfrag{4}{\tiny $3$}
\psfrag{u}{\tiny $u_*$}
\begin{tabular}{c}\epsfig{file=./pic/nv_p27.eps,width=0.2\linewidth}\end{tabular}
& 
\psfrag{0}{\tiny $0$}
\psfrag{1}{\tiny $1$}
\psfrag{2}{\tiny $2$}
\psfrag{3}{\tiny $3$}
\begin{tabular}{c}\epsfig{file=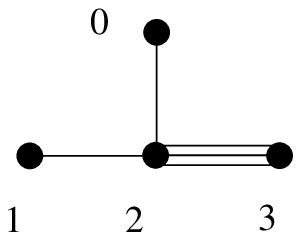,width=0.2\linewidth}\end{tabular}
&  $r_1r_2(v_0)$ \\

\hline
33& 
\psfrag{1}{\tiny $3$}
\psfrag{u}{\tiny $2$}
\psfrag{0}{\tiny $u_*$}
\begin{tabular}{c}\epsfig{file=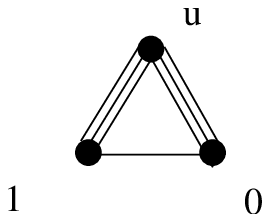,width=0.2\linewidth}\end{tabular}
& 
\psfrag{0}{\tiny $0$}
\psfrag{1}{\tiny $1$}
\psfrag{2}{\tiny $2$}
\psfrag{3}{\tiny $3$}
\begin{tabular}{c}\epsfig{file=./pic/nv_f31.eps,width=0.2\linewidth}\end{tabular}
&  $r_0r_1r_2r_3r_2(v_0)$ \\

\hline
34& 
\psfrag{0}{\tiny $3$}
\psfrag{1}{\tiny $2$}
\psfrag{u}{\tiny $u_*$}
\begin{tabular}{c}\epsfig{file=./pic/nv_p29.eps,width=0.2\linewidth}\end{tabular}
& 
\psfrag{0}{\tiny $0$}
\psfrag{1}{\tiny $1$}
\psfrag{2}{\tiny $2$}
\psfrag{3}{\tiny $3$}
\begin{tabular}{c}\epsfig{file=./pic/nv_f31.eps,width=0.2\linewidth}\end{tabular}
&  $r_1r_2r_3r_2(v_0)$ \\

\hline
35& 
\psfrag{0}{\tiny $2$}
\psfrag{1}{\tiny $1$}
\psfrag{u}{\tiny $u_*$}
\begin{tabular}{c}\epsfig{file=./pic/nv_p32.eps,width=0.2\linewidth}\end{tabular}
& 
\psfrag{0}{\tiny $0$}
\psfrag{1}{\tiny $1$}
\psfrag{2}{\tiny $2$}
\psfrag{3}{\tiny $3$}
\begin{tabular}{c}\epsfig{file=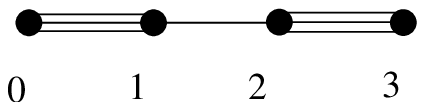,width=0.2\linewidth}\end{tabular}
&  $r_0r_1r_2r_3r_2(v_0)$ \\

\hline
\end{tabular}
\end{center}
\end{table}


\pagebreak
\clearpage

\section{Back to subalgebras}
\label{back to subalgebras}

In this section we show that each simplicial subgroup of an arithmetic over $\Q$ simplicial group corresponds to
some regular hyperbolic subalgebra of a hyperbolic Kac-Moody algebra.

Recall that a reflection group $K_1$ is {\rr}-isomorphic to a reflection group $K_2$ (we denote this by $K_1\r-equiv K_2$) 
if there is an isomorphism of the groups $K_1$ and $K_2$ preserving the set of reflections and parabolic elements.

As before, we denote by $\Delta_{\g}$ the root system of Kac-Moody algebra $\g$, and by $W(\Delta_{\g})$ the Weyl group of $\Delta_{\g}$.

\begin{theorem}
\label{subalg for each subgr}
Let $H\subset G$ be a simplicial subgroup of an arithmetic over $\Q$ simplicial group.
There exist a Kac-Moody algebra $\g$ and a regular subalgebra $\h\subset \g$ such that 
$W(\Delta_{\g})\r-equiv G$ and $W(\Delta_{\h})\r-equiv H$. 

\end{theorem}

\begin{proof}
If $\rank H=\rank G$ the theorem is shown in~\cite{T}.
If $\rank H<\rank G$ and $H\subset G$ is maximal, then the theorem coincides with 
Theorem~\ref{subalg for max subgr}.
%
%
%
%

To prove the theorem for non-maximal subgroups,
we are going to use Lemma~\ref{tower}, i.e. to show that for each subgroup $H\subset G$ one can choose a subgroup 
$H'\r-equiv H$ and a tower of embeddings
$$H'=K_0\subset K_1\subset K_2 \subset \dots \subset K_l=G$$ 
such that each single embedding $K_i\subset K_{i+1}$ corresponds to a subalgebra, where the algebra $\g_i$ with $W(\Delta_{\g_i})=K_i$
is the same for both embeddings $K_{i-1}\subset K_i$ and $K_i\subset K_{i+1}$.

Recall that in a non-simply-laced case a group $G$ can serve as a Weyl group for several root systems
(and an algebra is specified by the choice of the short and long vectors among the simple roots).
Notice that in case of a maximal subgroup satisfying  $\rank H< \rank G$, Theorem~\ref{subalg for max subgr} shows that
$\h\subset \g$ is a subalgebra for any algebra $\g$ having $G$ as a Weyl group of $\Delta_{\g}$
(independently of the choice of the short and long roots).

On the other hand, for the subgroups satisfying $\rank H= \rank G$ it is  shown in~\cite{T}
that it is possible to choose the algebra with Weyl group $G$ (i.e. to choose the set of long and the set of short simple roots in $\Delta$) such that the algebra $\h$ determined by the root system $\Delta_H$ turns into a subalgebra of $\g$
(and the other choices can easily lead to root systems which are not root subsystems).
We emphasize that maximality is not needed here.

So, we need to check that all the choices of lengths (for all $K_i\subset K_{i+1}$, $i=0,\dots,l-1$) agree. 
Using Tables~\ref{3}--\ref{6-5}
it is easy to check that for each subgroup $H\subset G$ one can find $H'\r-equiv H$ and a tower 
$$H'=H_0\subset K_1 \subset K_2 \subset...\subset K_l=G,$$ where each of $K_i\subset K_{i+1}$ but at most one 
are maximal subgroups satisfying $\rank K_i< \rank K_{i+1}$ 
(and the remaining embedding, if any, satisfies  $\rank K_i= \rank K_{i+1}$ but may not be maximal). 
Hence, for all steps but one we do not need to make any choices, and the choice made for the one step left 
leads to a subalgebra $\h_i\subset \h_{i+1}$ for each $i=0,\dots,l-1$. Now, applying  Lemma~\ref{tower}, we get the statement of the
 theorem.
 
\end{proof}

\begin{example}
Consider a subgroup $H\subset G$ shown on Fig~\ref{ex_tower}. A group $H$ can be included in the tower 
$$H=K_0\subset K_1\subset K_2 \subset K_3 =G,$$ where each of the embeddings $K_0\subset K_1$ and $K_2\subset K_3$
are of maximal rank. To obtain the corresponding subalgebra $\h\subset \g$, 
for each of these two embeddings we need to make 
an appropriate choice of the lengths of the roots, and the choices may contradict each other.
On the other hand, we may consider another tower $$H'=K_0'\subset K_1'\subset K_2'=G,$$
(where $H'\r-equiv H$, but the embedding $H'\subset G$ may differ from $H\subset G$).
The latter tower can be turned into the tower of subalgebras since for this tower we need to make the choice 
for the embedding $K_0'\subset K_1'$ only.

\end{example}

\begin{figure}[!h]
\begin{center}
\psfrag{a}{(a)}
\psfrag{b}{(b)}
\psfrag{H0}{\small $K_0=H$}
\psfrag{H1'}{\small $K_0\r-equiv H \r-equiv K_0'$}
\psfrag{H2}{\small $K_1$}
\psfrag{H3}{\small $K_2$}
\psfrag{H4}{}
\psfrag{1}{}
\psfrag{H3'}{\small $K_3=G=K_2'$}
\psfrag{H0'}{\small $K_0'=H'$}
\psfrag{H1}{}
\psfrag{H2'}{\small $K_1'$}
\epsfig{file=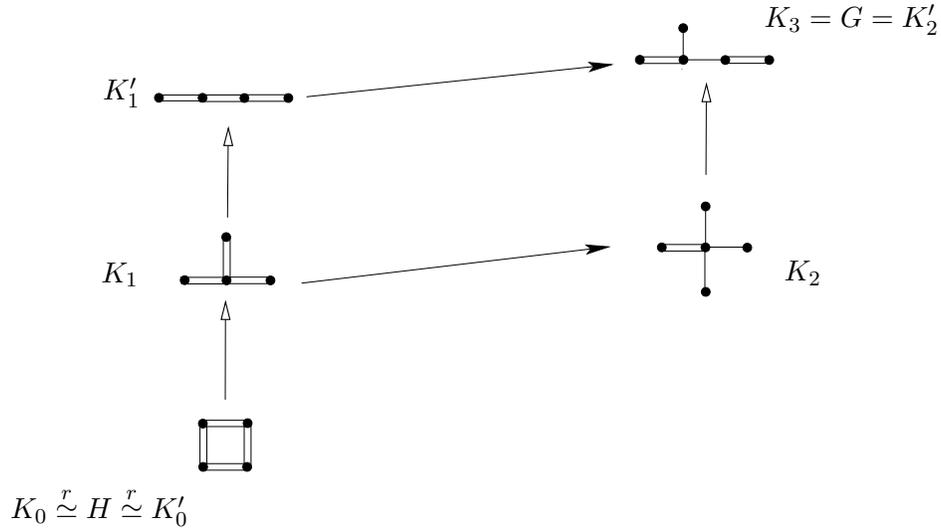,width=0.99\linewidth}
\caption{Two embeddings of the same group.}
\label{ex_tower}
\end{center}
\end{figure}

\begin{remark}
In fact, the algorithm provided in Section~\ref{algorithm}  gives way to classify all regular hyperbolic subalgebras
of hyperbolic Kac-Moody algebras. In view of Remark~\ref{class}, for each pair $(H,G)$ we can list all the embeddings $H\subset G$ 
(if any) up to inner isomorphisms of $G$. Given all embeddings $H\subset G$, construction of all subalgebras $\h\subset \g$ with 
$W(\Delta_{\h})=H$ and $W(\Delta_{\g})=G$ is immediate.

\end{remark}

In view of Theorem~\ref{subalg for each subgr} it is natural to ask the following question.

\begin{question}
Let $\g$ be an indefinite (but not hyperbolic) Kac-Moody algebra with root system $\Delta_{\g}$ and Weyl group $G=W(\Delta_{\g})$.
Given a reflection subgroup $H\subset G$, does there always exist a regular subalgebra $\h$ such that $H=W(\Delta_{\h})$? 

\end{question}

\end{document}